\documentclass{siamart1116}
\usepackage[utf8x]{inputenc}
\usepackage{graphicx}
\usepackage{amsmath}
\usepackage{amsfonts}
\usepackage{setspace}
\usepackage[position=below]{subfig}

\begin{document}

\title{A data-driven linear-programming methodology for optimal transport}
\author{{
\sc Weikun Chen, Esteban G. Tabak},\\[2pt]
Courant Institute of Mathematical Sciences, New York University\\
}
\maketitle

\begin{abstract}

A data-driven formulation of the optimal transport problem is presented and solved using adaptively refined meshes to decompose the problem into a sequence of finite linear programming problems. Both the marginal distributions and their unknown optimal coupling are approximated through mixtures, which decouples the problem into the the optimal transport between the individual components of the mixtures and a classical assignment problem linking them all.
A factorization of the components into products of single-variable distributions makes the first sub-problem solvable in closed form. The size of the assignment problem is addressed through an adaptive procedure: a sequence of linear programming problems which utilize at each level the solution from the previous coarser mesh to restrict the size of the function space where solutions are sought.
The linear programming approach for pairwise optimal transportation, combined with an iterative scheme, gives a data driven algorithm for the Wasserstein barycenter problem, which is well suited to parallel computing.

\end{abstract}

\section{Introduction}

\label{sec:1}
The optimal transport problem has received a considerable interests in recent years, due in part to its wide scope of applicability in fields that include econometrics, data analysis, fluid dynamics, automatic control, transportation, statistical physics, shape optimization, expert systems and meteorology \cite{Rac98,Vil03}.

In a modern formulation of Monge's original statement of the optimal transport problem, one has two probability density functions $\rho(x)$ and $\mu(y)$, $x, y \in\mathbb{R}^d$. A map $M$ from $\mathbb{R}^d$ to $\mathbb{R}^d$ is said to preserve mass (or to push forward $\rho$ into $\mu$) if, for all bounded subsets $S \subset \mathbb{R}^d$,
\begin{equation}
    \int_{x\in S}\rho(x)\ dx = \int_{y\in M(S)}\mu(y)\ dy \ .
\end{equation}
For smooth one-to-one maps, this yields the differential condition
\begin{equation}
   \left| \mathrm{det}(\nabla M(x)) \right| \mu(M(x))=\rho(x)\ .
\end{equation}
Among all mass preserving maps, one seeks the \emph{optimal map} that minimizes a cost such as the Wasserstein distance
\begin{equation}
    W_p(\rho(x),\mu(y))= \left(\inf_{M}\int\|M(x)-x\|^p\rho(x)\ dx\right)^{1/p},
\end{equation}
\\
where $p\geq1$ is fixed.

Monge's formulation, which seeks a map $M(x)$ connecting $\rho(x)$ and $\mu(y)$, is  nonlinear and does not necessarily have unique solutions. In addition, many applications do not require the transport between $\rho(x)$ and $\mu(y)$ to be one-to-one.
Kantorovich \cite{Kan48} proposed a relaxation where the map $y=M(x)$ is replaced by a more general \emph{coupling} $\pi(x,y)$:
\begin{equation}\label{primal}
\begin{aligned}
&\min_{\pi(x,y)} && \int c(x,y)\pi(x,y)\ dx dy \\
&s.t && \int \pi(x,y)\ dx=\mu(y) \\
&    && \int \pi(x,y)\ dy=\rho(x) \\
&    && \pi(x,y)\geq 0 .
\end{aligned}
\end{equation}
Here $c(x,y)$ is a pointwise cost function such as $\|y-x\|^p$, and the optimization is performed over the set of joint distributions $\pi(x,y)$ with marginals $\rho(x)$ and $\mu(y)$. It has been shown that, under suitable conditions on the cost function and the marginals, the solution to Kantorovich's problem also solves Monge's, i.e. the coupling $\pi(x,y)$ is concentrated on the graph of a function $y=M(x)$ \cite{Gan96}.

Among the methods that have been proposed to solve the optimal transport problem numerically, Benamou and Brenier \cite{Ben00} proposed a fluid mechanical framework, constructing an optimal path from $\rho(x)$ to $\mu(y)$ by solving an optimization problem with a partial differential equation as a constraint. Haber, Rehman and Tannenbaum \cite{Hab10} proposed a modification of the objective function, discretized both the objective function and the constraints and solved the problem using sequential quadratic programming. Marco Cuturi \cite{Cut13} introduced an iterative way of solving the optimal transportation problem by adding an entropic regularization term. The regularized problem can be solved by a fix-point iterative algorithm. Adam Oberman and Yuanlong Ruan used sums of delta functions to discretize the problem into linear programming problems and solved the problem using grid refinement \cite{Adam15}. The methodology developed in this article extends this latter approach by allowing more general models for both the marginal distributions and their coupling, and optimizing them to a higher degree.
Other numerical procedures can be found in \cite{Obe08,Dea06,Cha09,Ang03}.

Most of these procedures assume that the marginal distributions $\rho(x)$ and $\mu(y)$ are known explicitly. Yet this is not the case in the great majority of applications, where these distributions are only known through a finite set of samples $\{x_i\}$ and $\{y_j\}$ from $\rho(x)$ and $\mu(y)$ respectively. A formulation of the optimal transportation problem in terms of samples was proposed in \cite{Tab15}, which also developed a methodology for its numerical solution following a gradient flow in feature-space that pushes $\rho(x)$ to the target distribution $\mu(y)$ through a time-dependent map $z(x; t)$, with $z(x; 0)=x$ and $z(x; \infty) = y(x)$. Another formulation for the sample-based optimal transport and barycenter problems based on local solvers and feature functions was proposed in \cite{Max17b}, and a sample-based preconditioning technique was developed in \cite{Max17}.

In this article, we propose an alternative methodology for solving the sample-based optimal transportation problem, based on a ``divide and conquer'' methodology, where the marginal distributions are approximated through mixtures of elementary distributions with support on the individual cells of a rectangular grid. By modeling the coupling also as a mixture of distributions gives rise to two subproblems: the optimal pairwise transport between the individual components of the mixtures and a global linear assignment problem linking these various local optimal solutions. When the components are given by the product of one-dimensional densities, the first sub-problem can be solved in closed form. The rectangular grids are sequentially refined, with the space of allowable solutions limited to a proper size using considerations involving the solution of the previous, coarser problem.
The formulation developed here differs from ordinary rectangle discretization in three main ways: a) it allows more general approximations to the marginal distributions,  b) under such approximations, the constraints are satisfied exactly, and c) the components of the coupling are chosen optimally rather than as the cartesian product of the components of the marginals. Since these extensions turn out not to increase the size of the problem, they yield a higher order accuracy at no additional cost.

The outline of the paper is as follows. Section 2 develops the ``divide and conquer'' strategy and its two subproblems.
Section 3 further develops the algorithm to refine the rectangular meshes and update the functional family where solutions are sought at each stage using the previous solution. In addition, specifics are provided on the family of functions used to approximate the marginal distributions. Section 4 contains some numerical results using the pairwise optimal transport solver developed. Section 5 combines our methodology with an iterative scheme to solve the Wasserstein barycenter problem. Two applications are discussed in section 6 and a summary is provided in section 7.

\section{A ``divide and conquer'' strategy}

We start with the optimal transport problem in Kantorovich formulation:
\begin{equation}\label{primalG}
\begin{aligned}
&\min && \int c(x,y)\ \pi(x,y)\ dxdy \\
&s.t  && \int \pi(x,y)\ dx=\mu(y) \\
&     && \int \pi(x,y)\ dy=\rho(x) \\
&     &&\pi(x,y)\geq 0
\end{aligned}
\end{equation}
and seek to re-formulate it and solve it in a way that adapts to situations where one only knows the source and target distributions $\rho(x)$ and $\mu(y)$ through two sets of independent samples: $\left\{x_i\right\} \sim \rho$, $\left\{y_j\right\} \sim \mu$, $i \in [1, m]$, $j \in [1,  n]$.

Often the two marginal distributions  can be estimated through mixtures of the form
\begin{equation}
 \rho(x) = \sum_i p_i\rho_i(x), \quad \mu(y) = \sum_j q_j \mu_j(y),
 \label{approx_marginals}
\end{equation}
where the $p_i, q_j$ are non-negative weights that add up to one, and the $\rho_i(x), \mu_j(y)$ are independent probability distributions. Examples include:
\begin{itemize}
\item The empirical distributions
$$ \rho(x) = \frac{1}{m} \sum_i \delta(x-x_i), \quad \mu(y) = \frac{1}{n} \sum_j \delta(y-y_j), $$

\item The kernel density estimations
$$ \rho(x) = \frac{1}{m} \sum_i K(x-x_i), \quad \mu(y) = \frac{1}{n} \sum_j K(y-y_j), $$
with specified kernels $K$ , such as isotropic Gaussians with given bandwidths,

\item Gaussian mixtures, where the  $\{\rho_i(x)\}, \{\mu_j(y)\}$ are Gaussians distributions with means, covariance matrices and weights estimated from the data through the EM procedure,

\item Histograms, i.e. piecewise-constant approximations to $\{\rho_i(x)\}, \{\mu_j(y)\}$ estimated by counting samples within each cell of a grid,

\item Generalizations of the above, such as piecewise linear distributions and locally defined products of single-variable functions (we will discuss the latter in more depth below.)

\end{itemize}

At a similar level of approximation, one may propose to model the unknown coupling by a distribution of the form:
\begin{equation}
 \pi(x,y) = \sum_{i,j} \lambda^{i,j} \pi_{ij}(x,y),
 \label{pi_sum}
\end{equation}
where the $\pi_{ij}(x,y)$ satisfy the marginal constraints
\begin{equation}
\int \pi_{ij}(x,y)\ dy = \rho_i(x), \quad \int \pi_{ij}(x,y)\ dx = \mu_j(y).
\label{ij_marginals}
\end{equation}
(Distributions $\pi_{ij}$ satisfying (\ref{ij_marginals}) always exist, the simplest example consisting of the products $\pi_{ij}(x,y) = \rho_i(x)\ \mu_j(y)$.)
Then we have
$$ \int \pi(x,y) dy =  \sum_{i,j} \lambda^{i,j} \rho_i(x), \quad \int \pi(x,y) dx =  \sum_{i,j} \lambda^{i,j} \mu_j(y), $$
so the constraints in (\ref{primalG}) with marginals from (\ref{approx_marginals}) are satisfied exactly by any solution $\lambda^{i,j}$ of the system
\begin{equation}\label{marginal_constraints}
\begin{aligned}
& \sum_{j} \lambda^{i,j}   =  p_i \\
& \sum_{i} \lambda^{i,j}   =  q_j \\
& \lambda^{i,j} \ge 0.
\end{aligned}
\end{equation}
Moreover, the first two sets of conditions are also necessary when the distributions $\rho_i(x)$ and $\mu_j(y)$ are independent, while positivity is a natural requirement for a mixture of distributions.

\begin{theorem}\label{generalTheorem}
Problem (\ref{marginal_constraints}) always has solutions when $p_i\geq 0$, $q_j\geq 0$ and $\sum_i p_i = \sum_j q_j$.
\end{theorem}
\noindent
\emph{\bf Proof:} Introduce $S=\sum_i p_i = \sum_j q_j > 0$ ($S=1$ in our case of interest, since the $p_i$ and $q_j$ are weights in a mixture of distributions.) One can directly construct an explicit solution to (\ref{marginal_constraints}) by setting $\lambda^{i,j}=\frac{p_i  q_j}{S}$.

\medskip

It follows that the following approximation to (\ref{primalG}):
\begin{equation}\label{relaxed_primal_2}
\begin{aligned}
&\min_{\lambda, \pi} && \sum_{i,j} \lambda^{i,j} \iint c(x,y) \pi_{ij}(x,y) \ dxdy \\
&s.t  && \sum_{j} \lambda^{i,j}   =  p_i \\
&     && \sum_{i} \lambda^{i,j}   =  q_j. \\
&     && \int \pi_{ij}(x,y) \ dy = \rho_i(x) \\
&     && \int \pi_{ij}(x,y) \ dx = \mu_j(y) \\
&     &&\lambda^{i,j}\geq 0, \pi_{ij}(x,y) \ge 0
\end{aligned}
\end{equation}
is always feasible. Moreover, in the optimal solution to (\ref{relaxed_primal_2}), each $\pi_{ij}(x,y)$ is the optimal plan with marginals $\rho_i(x)$ and $\mu_j(y)$, so this optimization problem can be decomposed into the sub problems

\begin{equation}\label{sub_1}
\begin{aligned}
C_{ij} :=\  &\min_{\pi_{ij}} && \iint c(x,y) \pi_{ij}(x,y) \ dxdy \\
&s.t  && \int \pi_{ij}(x,y) \ dy = \rho_i(x) \\
&     && \int \pi_{ij}(x,y) \ dx = \mu_j(y) \\
&     &&\pi_{ij}(x,y)\geq 0
\end{aligned}
\end{equation}
and
\begin{equation}\label{sub_2}
\begin{aligned}
&\min_{\lambda} && \sum_{i,j} C_{ij}\lambda^{i,j} \\
&s.t  && \sum_{j} \lambda^{i,j}   =  p_i \\
&     && \sum_{i} \lambda^{i,j}   =  q_j. \\
&     &&\lambda^{i,j}\geq 0,
\end{aligned}
\end{equation}
both of which have feasible solutions.

\medskip.

\noindent
{\bf Remarks:}
\begin{itemize}
\item When $\rho_i = \delta(x-x_i)$, $\mu_j = \delta(y-y_j)$, $p_i=1/m$, $q_j=1/n$ --i.e. for the empirical marginal distributions-- we have that $\pi_{ij} = \delta(x-x_i) \delta(y-y_j)$ and $C_{ij} = c(x_i,y_j)$, and the problem reduces to the optimal assignment between the $\{x_i\}$ and $\{y_j\}$.

\item If one introduces regular grids in $x$ and $y$ space, sets the $\rho_i$ and $\mu_j$ as the normalized characteristic functions for cells $i$ and $j$ in these grids, and instead of solving (\ref{sub_1}) one uses the sub-optimal $\pi_{ij} = \rho_i \mu_j$ and defines $C_{ij} = c(\bar{x}_i,\bar{y}_j)$, where $\bar{x}_i, \bar{y}_j$ are the centers of the corresponding cells, one recovers the linear-programming procedure in \cite{Adam15}.
\end{itemize}

Assessing the formulation in (\ref{sub_1}, \ref{sub_2}) involves two kinds of considerations: how well its solution (\ref{pi_sum}) approximates the true optimal solution $\pi(x,y)$ to (\ref{primalG}) --its ''quality'', and how difficult and computationally expensive it is to find this solution --its computability.

The quality of the solution depends on two factors: the accuracy of the proposed estimation of the marginal distributions through mixtures in (\ref{approx_marginals}), and the locality of the components $\rho_i(x), \mu_j(y)$: if the optimal map establishes a one-to-one correspondence between them --i.e. if the $\lambda_{ij} \in \{0,1\}$-- then the optimal $\pi(x,y)$ agrees with its estimation in (\ref{pi_sum}) if the approximation (\ref{approx_marginals}) is exact. On the other hand, if a $\rho_i(x)$ is mapped onto various $\mu_j(y)$, true-optimality can no longer be guaranteed.

Computability, on the other hand, depends on three distinct factors: the size of the linear programming problem in (\ref{sub_2}) (i.e. the number of unknown $\lambda_{ij}$), the difficulty of the density estimation leading to (\ref{approx_marginals}), and the difficulty of the solution to each individual optimal transport problem in (\ref{sub_1}).

Addressing these considerations leads to an effective algorithm:
\begin{itemize}
\item In order to enforce the locality of the $\rho_i, \mu_j$, one would like these to have compact, concentrated and non-overlapping support. Ideally, each $\rho_i$ and $\mu_j$ should have support in an area of small diameter, yet large enough to contain a number of sample points permitting a robust density estimation.

\item The requirement above could lead to overwhelmingly large linear programming problems to solve. To avoid this, one needs to reach this point having restricted the number of unknowns, i.e. the $\lambda^{ij}$ with potentially nonzero values. This can be achieved by an adaptive refinement procedure similar to the one developed in \cite{Adam15}. Adaptive refinement is easiest for rectangular grids, which suggests proposing $\rho_i, \mu_j$ with support in disjoint rectangular cells.

\item One way to achieve inexpensive density estimations for the marginal distributions suggests modeling each $\rho_i(x)$ and $\mu_i(y)$ as products of one dimensional probability densities, whose estimation is comparatively easy.

\item As shown by theorem \ref{optimalPlanOfIndependentFunctions} below, this proposal for the $\rho_i, \mu_j$ also provides analytical solutions to the subproblems in (\ref{sub_1}), thus reducing the computation to a sequence of discrete assignment problems.

\end{itemize}

\begin{theorem}\label{optimalPlanOfIndependentFunctions}
Consider the optimal transport problem (\ref{primalG}) in $\mathbb{R}^d$ with cost function $c(x,y) = \frac{1}{2}\|y-x\|^2$, and assume that the source and target probability distributions factorize into one-dimensional densities:
$$ \rho(x) = \prod_{l=1}^d \rho_l(x_l), \quad \mu(y) = \prod_{l=1}^d \mu_l(y_l).$$
Denote the optimal plan and optimal map between $\rho_l(x_l)$ and $\mu_l(y_l)$ by $\pi_l(x_l,y_l)$ and $m_l(x_l)$ respectively. Then
\begin{enumerate}
\item The optimal plan between $\rho(x)$ and $\mu(y)$ is given by
$$\pi(x,y) = \prod_{l=1}^d \pi_l(x_l,y_l), $$
\item The optimal map pushing forward $\rho(x)$ to $\mu(y)$ is given by
$$ y=m(x) = (m_1(x_1),\ldots,m_d(x_d))^T, $$
\item The optimal value $C$ of the transport problem between $\rho(x)$ and $\mu(y)$ equals the sum of the optimal values $C_l$ of the transport problems between $\rho_l(x_l)$ and $\mu_l(y_l)$:
$$ C = \sum_{l=1}^d C_l . $$
\end{enumerate}
\end{theorem}

\begin{proof}

Because the $\pi_l(x_l,y_l)$ have marginals $\rho_l(x_l))$ and $\mu_l(y_l)$, it follows that $\pi(x,y) = \prod_{l=1}^d \pi_l(x_l,y_l)$ has marginals $\rho(x)$ and $\mu(y)$, and constitutes therefore a feasible solution. For instance,
$$ \int \pi(x,y) \ dy = \int \left(\prod_{l=1}^d \pi_l(x_l,y_l)\right) \ dy_1 \ldots dy_d = \prod_{l=1}^d \rho_l(x_d) = \rho(x). $$
Furthermore, because of the optimality of the $\pi_l(x_l,y_l)$, we can find dual functions $\phi_l(x_l)$, $\psi_l(y_l)$ such that:
\begin{equation}\label{dualConstraints}
 \phi_l(x_l) + \psi_l(y_l) \leq \frac{1}{2}(x_l-y_l)^2
\end{equation}
and
\begin{equation}\label{dualityGap}
\begin{aligned}
C_l = \iint \frac{1}{2}(x_l-y_l)^2\pi_l(x_l,y_l) dx_ldy_l\\
 = \int \phi_l(x_l)\rho_l(x_l) dx_l + \int \psi_l(y_l)\mu_l(y_l) dy_l.
\end{aligned}
\end{equation}
Introducing
$$ \phi(x) = \sum_{l=1}^d \phi_l(x_l), \quad \psi(y) = \sum_{l=1}^d \psi_l(y_l), $$
we obtain, adding up over $l$ the inequalities in (\ref{dualConstraints}), that
$$
 \phi(x) + \psi(y) \leq \frac{1}{2}\|x-y\|^2.
$$
On the other had, we have
$$
\begin{aligned}
&C = \iint \frac{1}{2}\|x-y\|^2\pi(x,y) dx dy = \iint \left(\sum_{l=1}^d \frac{1}{2}(x_l-y_l)^2\right) \pi(x,y) dx dy\\
& = \sum_{l=1}^d \iint \frac{1}{2}(x_l-y_l)^2 \pi_l(x_l,y_l) \prod_{h\ne l} \pi_h(x_h,y_h) dx dy & \\
& = \sum_{l=1}^d \iint \frac{1}{2}(x_l-y_l)^2 \pi_l(x_l,y_l) dx_l dy_l = \sum_{l=1}^d  C_l\\
& = \sum_{l=1}^d  \int \phi_l(x_l) \rho_l(x_l) dx_l + \int \psi_l(y_l) \mu_l(y_l) dy_l &\\
&= \sum_{l=1}^d  \int \phi_l(x_l) \rho(x) dx + \int \psi_l(y_l) \mu(y) dy  = \int \phi(x) \rho(x) dx + \int \psi(y) \mu(y) dy,&
\end{aligned}
$$
which proves the optimality of the feasible solution $\pi(x,y)$ and, in addition, that $C = \sum_{l=1}^d C_l$ and that the $\phi, \psi$ so defined constitute the optimal solution of the dual problem. It follows that
\begin{equation}
y(x) = \nabla\phi(x) = \left(\frac{d\phi_1(x_1)}{dx_1},\ldots, \frac{d\phi_d(x_d)}{dx_d}\right)^T = (m_1(x_1),\ldots,m_d(x_d))^T,
\end{equation}
which finishes the proof.
\end{proof}

This theorem addresses simultaneously two of the considerations above: estimating the one-dimensional marginal distributions $\rho_{i,l}(x_l), \mu_{j,l}(y_l)$ is comparatively easy (one simply disregards all other components of the samples; we will propose some simple forms for these one-dimensional estimates below), and finding the optimal map $m_{ij,l}(x_l)$ between the one dimensional distributions $\rho_{i,l}(x_l)$ and $\mu_{j,l}$ is straightforward:
\begin{equation}
 m_{ij,l}(x_l) = Q_{j,l}^{-1} \left(P_{i,l}(x_l)\right),
 \label{opt_1d}
\end{equation}
where $P_{i,l}$ and $Q_{j,l}$ are the cumulative distributions of $\rho_{i,l}$ and $\mu_{j,l}$ respectively.
With these tasks performed, the whole problem turns into solving the assignment problem in (\ref{sub_2}).

\section{The procedure}

 From the discussion above, we need to solve the linear programming problem (\ref{sub_2}) with costs $C_{ij}$ given by the solutions to  (\ref{sub_1}), where the $\rho_i$  have support on cell $i$ of a rectangular grid and consist of products of one dimensional distributions $\rho_i^l(x_l)$, and similarly for $\mu_j$. In order to implement an algorithm along these lines, we need to specify the following:
\begin{enumerate}
\item The grid.

\item The estimation procedure for the $p_i$, $q_j$.

\item The form of the proposed one dimensional distributions $\rho_i^l$, $\mu_j^l$ and the procedure to estimate them.
\end{enumerate}

Yet the grid cannot be set at once in its final form, as the number of unknowns $\lambda^{ij}$ would be very large. To bypass this constraint and yet achieve high resolution, we use an adaptive refinement process similar to the one developed in \cite{Adam15}.
At each refinement step, the question is how to use the solution on the coarser grid to reduce the set of available $\lambda^{ij}$ on the finer grid,  while keeping the problem feasible and increasing the accuracy of the solution.
The tradeoff here is that, if we consider a large set of $\lambda$'s, the problem will become intractable soon, while if the set it too small, the problem may become infeasible or its solution may be suboptimal from the perspective of the full problem. In the procedure that follows, we provide a process for refinement that guarantees feasibility while addressing accuracy:

\begin{enumerate}
\item Start with a very coarse grid, where the support of each component of the $\{x_i\}$ is divided into two segments, so we end up with $2^d$ cells in the grid (Here and below, whenever we describe a procedure as it applies to the $x$, $\{x_i\}$, $\rho_i$, the same procedure is implied for the $y$, $\{y_j\}$, $\mu_j$.) We set the initial $S$ as the set of all possible pairs $(i,j)$ between the two grids, with cardinality $|S| = 2^{2d}$.

\item Assign the weights $p_i$ simply as the fraction of samples in cell $i$.

\item Estimate the $\rho_i^l$ (proposals are described below), compute the $C_{ij}$ and solve \ref{sub_2}, where only the $\lambda_{ij}$ with $(i,j) \in S$ are allowed to be nonzero.

\item Divide each segment containing more than $n_{min}$ samples $\{x_i\}$ into subsegments, and use the procedure described below to update $S$.

\item Return to step 2.

\end{enumerate}

\subsection{The refinement step}

The refinement step consists of two sub-steps: partitioning each segment containing enough samples to create a finer grid, and deciding which subset $S$ of all possible cells $(i,j)$ in the new grid is eligible for candidate nonzero values of $\lambda^{ij}$.

\subsubsection{Segment partition}

A basic ingredient of the adaptive scheme is the division of each cell into two or more subcells, which we choose to do dimension by dimension. A number of choices arise on when, into how many and how to divide a given segment; the other components of the algorithm are blind to these choices.

Regarding ``when'', one may chose to, at each step, divide all segments. Alternatively, one may choose to divide only those segments that contain sufficient samples points. For instance, one may order the segments along each dimension by number of samples, and divide only those in the first quartile.

As to the ``into how many'', we have adopted binary divisions, but other choices may be made. For instance, in combination with the choice above on when to partition, we may subdivide more finely those segments with the larger number of samples.

Finally on the ``how'', one may partition each segment into sub-segments of the same length or with the same number of samples, or a weighted balance between the two.

In the experiments of section \ref{sec1.4}, we have adopted the choice of always dividing each segment into two sub-segments of equal length, except in the high-dimensional examples, where we adopt a different choice described in that section.

\subsubsection{The minimal set}

Each cell $i$ in the newly divided grid has a parent cell $h$ in the coarser grid of the prior step. We first consider the set $S_{min}$ consisting of all those pairs of cells whose pairs of parent cells $(h,k)$ have nonzero $\lambda^{hk}$ in the prior solution. This choice guarantees feasibility of (\ref{sub_2}), as shown by the following theorem:

\begin{theorem}\label{Feasibility}
There exists a set of $\lambda^{ij}$ satisfying the constraints (\ref{sub_2}) such that $\lambda^{ij} = 0$ whenever $\Lambda^{hk}=0$, where $(h,k)$ are the parent cells of $(i,j)$ and $\Lambda^{hk}$ denotes a feasible solution of (\ref{sub_2}) on the coarser grid.

\end{theorem}

\begin{proof}
Introduce the following notation:
$$ h(i), k(j): \hbox{parent cells of $(i,j)$}, $$
$$ I_h, J_k : \hbox{Sets of $\{i\}$, $\{j\}$ that have $h$, $k$ as parent cells},$$
$$ P_h, \ Q_k, \ \Lambda^{hk}: \hbox{weights and feasible solution in the coarser grid}.$$
Then, because of the choice of selecting the weights by counting samples,
$$ \sum_{i\in I_h} p_i = P_h, \quad \sum_{j\in J_k} q_j = Q_k,$$
and we can satisfy the constraints in (\ref{sub_2}) with nonzero values of $\lambda^{ij}$ in $S_{min}$ given by
$$ \lambda^{ij} = \frac{p_i q_j}{P_{h(i)} Q_{k(j)}} \Lambda^{hk}, $$
since
$$ \sum_j \lambda^{ij} = \sum_k \sum_{j\in J_k} \lambda^{ij} =  \sum_k \frac{p_i}{P_{h(i)}} \Lambda^{hk} = p_i$$
and
$$ \sum_i \lambda^{ij} = \sum_h \sum_{i\in I_h} \lambda^{ij} =  \sum_h \frac{q_j}{Q_{k(j)}} \Lambda^{hk} = q_j.$$
\end{proof}

\subsubsection{Additional cells}

The set $S_{min}$ described in the prior subsection guarantees feasibility, but may be suboptimal, as the true optimal solution in the finer grid may include transfer between cells whose parents did not interact at the coarser level. Yet one expects, from the  smoothness of the optimal maps, that such new interactions will involve cells in the immediate neighborhood of those in $S_{min}$. This suggest the following strategy for enlarging the support from $S_{min}$ to a set that, while larger, has treatable size:

For each $(i,j)$ in $S_{min}$, find the neighbors $I_i$ of the $i$th x-rectangle and the neighbors $J_j$ of the $j$th y-rectangle, and add $(i, J_j)$ and $(I_i, j)$ to $S_{min}$. Here the neighborhood $I_i$ is defined as the set of rectangles that have at least one point in common with rectangle $i$.

\subsection{A simple class of one-dimensional distributions}

In order to complete the algorithm's description, it only remains to specify the estimation of the $\rho_i^l$ with support in the interval $\left[x_{left}^{i,l}, x_{right}^{i,l}\right]$ from the subset of $l$-component of the samples $\{x_k\}$ that lie in that interval and, mutatis mutandis, the estimation of the $\mu_j$. Since the optimal map between any one-dimensional $\rho_i(x)$ and $\mu_j(y)$ is given straightforwardly by (\ref{opt_1d}), the choices here are quite broad. Two extreme scenarios in terms of complexity are the uniform
$$ \rho_i^l(x^l) =\begin{cases}  \frac{1}{\left[x_{right}^{i,l}- x_{left}^{i,l}\right]} & \hbox{for $x^l \in \left[x_{left}^{i,l}, x_{right}^{i,l}\right]$} \\
0 & \hbox{elsewhere} \end{cases},$$
which requires no estimation and yields a $d$-dimensional histogram as an estimation for the full marginal $\rho(x)$, and the piecewise linear, empirically defined $\rho_i^l(x^l)$, with constant slope in the interval between any two consecutive points equidistant between two consecutive sorted samples.

For the examples in this article, we have made the second-simplest choice of a linear $\rho_i^l(x^l)$:
$$ \rho_i^l(x^l) =\begin{cases}  \frac{1 + a \left(x^l-\bar{x}_i^l \right)}{\left[x_{right}^{i,l}- x_{left}^{i,l}\right]} & \hbox{for $x^l \in \left[x_{left}^{i,l}, x_{right}^{i,l}\right]$} \\
0 & \hbox{elsewhere,} \end{cases}$$
where
$$ \bar{x}_i^l = \frac{x_{right}^{i,l}+x_{left}^{i,l}}{2}, $$
$$ a = \min(\max(a_0, -b), b), \quad a_0 =  \frac{4 \sum_k \left(x_k^l -  \bar{x}_i^l\right)}{n_k \left[x_{right}^{i,l}- x_{left}^{i,l}\right]}, \quad b = \frac{2}{\left[x_{right}^{i,l} - x_{left}^{i,l}\right]} ,$$
where the $n_k$ values $x_k^l$ are the $l$-components of those samples that lie in the interval considered. The definition of $a$ above follows from imposing that the mean should agree with the empirical one when this does not violate the positivity of $\rho$.

\section{Numerical results}

\label{sec1.4}
\subsection{Measures of performance}
In order to assess the numerical experiments that follow, we design some quality measures.
\subsubsection{Optimal Map error}

The optimal map follows from the optimal plan through
\begin{equation}\label{optimalMap}
y(x) = \frac{\int y\pi(x,y)dy}{\int\pi(x,y)dy} = \frac{\sum_{ij}\lambda_{ij}\int y\pi_{ij}(x,y)dy}{\sum_{ij}\lambda_{ij}\rho_i(x)}
      = \frac{\sum_{ij}\lambda_{ij}\rho_i(x)m_{ij}(x)}{\sum_{ij}\lambda_{ij}\rho_i(x)},
\end{equation}
where $m_{ij}(x)$ denotes the optimal map between $\rho_i(x)$ and $\mu_j(y)$.

A first measure, when the analytical solution to the problem is known, is to compute the expected L-2 error between the numerical optimal map from Equation \ref{optimalMap} and the analytical solution with respect to the first marginal distribution:

\begin{equation}\label{error1}
\begin{aligned}
E_1&=\left(\int(y(x)-\bar{y}(x))^2\rho(x)dx\right)^{\frac{1}{2}}\\
  &\approx\left(\frac{1}{N}\sum_{l}(y(x_l)-\bar{y}(x_l))^2\right)^{\frac{1}{2}},
\end{aligned}
\end{equation}
where $\bar{y}(x)$ is the analytical solution to the problem.

Similarly, we can find a measure with respect to the second marginal distribution. Combining the two errors gives a first measure of performance.
\subsubsection{Error in the Wasserstein distance.}

The second measure that we use measures the difference between the numerical Wasserstein distance $W$ and the real distance $\tilde{W}$, when this is known:

\begin{equation}
E_2 = W - \tilde{W}.
\end{equation}
%

\subsection{Numerical experiments}

\subsubsection{Plan between 2-D Gaussian distributions}

We first compute the optimal plan between two 2-D Gaussian distributions known through samples. The first Gaussian has mean $m_1=(0,0)$ and covariance matrix $\Sigma_1 = \left(\begin{array}{cc}4 & -1\\-1 & 1\end{array}\right)$ known through 100000 data points. The second Gaussian has mean $m_2 = (0,0)$ and covariance matrix $\Sigma_2 = \left(\begin{array}{cc}9 & 8\\8 & 9\end{array}\right)$ known through 100000 data points.

The reasons for us to use so many samples are: 1) to test the efficiency of the method in dealing with large data sets, and 2) to keep the problem nontrivial after the refining process goes through several steps.

The analytical solution between two Gaussian distributions is given by the linear map:
\begin{equation}
y(x) = \Sigma_2^{1/2}(\Sigma_2^{1/2}\Sigma_1\Sigma_2^{1/2})^{-1/2}\Sigma_2^{1/2}(x-m_1)+m_2.
\end{equation}

The comparison between different methods in map error and objective value error is shown in figures, \ref{figureMapErrorX}, \ref{figureMapErrorY} and \ref{figureCrossterm}. Our linear programming method gives a better optimal map error, as shown in Figures \ref{figureMapErrorX} and \ref{figureMapErrorY}.

Notice that, in Figure \ref{figureCrossterm}, the objective value error for the ordinary LP method is negative --i.e., better than optimal!-- at the first several steps. The objective value error is measured as \[\emph{numerical} - \emph{real optimum},\] and the implication is that, for the first several steps, the ordinary LP method does not fully satisfy the marginal constraints of the true solution. Thus its ``lower'' cost is gained at the cost of violating the constraints.

\begin{figure}
\centering
  \includegraphics[width=13cm,height=8.5cm]{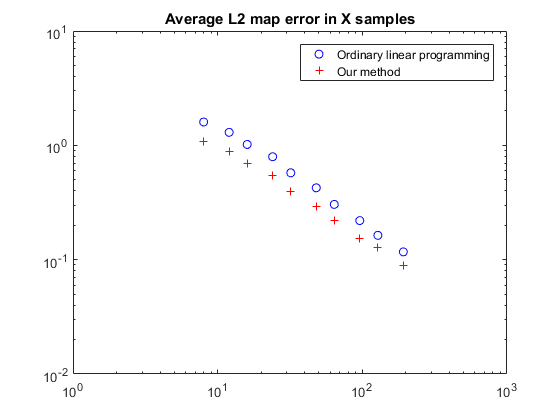}
\caption{Comparison of the optimal map error for each refinement step between two different formulations, in terms of the marginal $X$. The $x$-axis represents the number of cells per dimension.}
\label{figureMapErrorX}
\end{figure}

\begin{figure}
\centering
  \includegraphics[width=13cm,height=8.5cm]{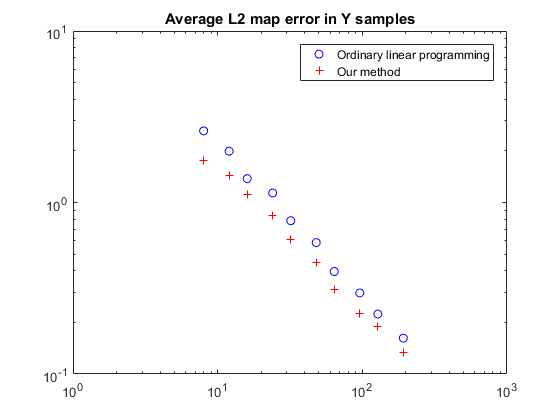}
\caption{Comparison of the optimal map error for each refinement step between different formulations, in terms of the marginal $Y$. The $x$-axis represents the number of cells per dimension.}
\label{figureMapErrorY}
\end{figure}

\begin{figure}
\centering
  \includegraphics[width=12cm,height=8.5cm]{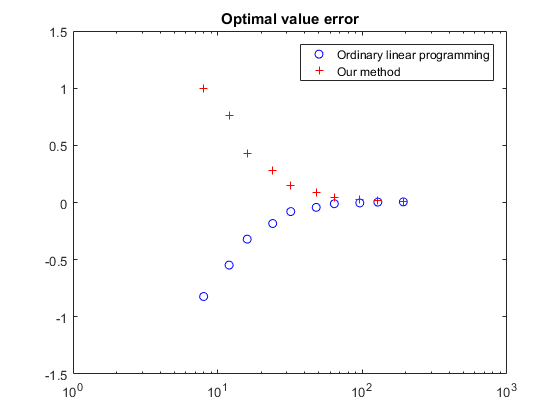}
\caption{Error in the value of the objective function for each refinement step.}
\label{figureCrossterm}
\end{figure}

\subsubsection{2-D Gaussian Mixture}
The next example finds the optimal map between two Gaussian mixtures, each known through 100000 data points.

The first mixture consists of 2 Gaussian distributions with weights (0.5, 0.5), means (4,0), (-4,0) and covariance matries $\left(\begin{array}{cc}1 & -0.5\\-0.5 & 1\end{array}\right)$ and $\left(\begin{array}{cc}9 & 2\\2 & 4\end{array}\right)$.

The second Gaussian mixture is artificially created by using another sample set of the first Gaussian mixture and applying a linear map, which by construction is the optimal map between the two mixtures.

A comparison of the optimal map found through regular LP and our method is shown in Figure \ref{figureMixtureMap}.

\begin{figure}
\centering
  \includegraphics[width=6cm,height=5cm]{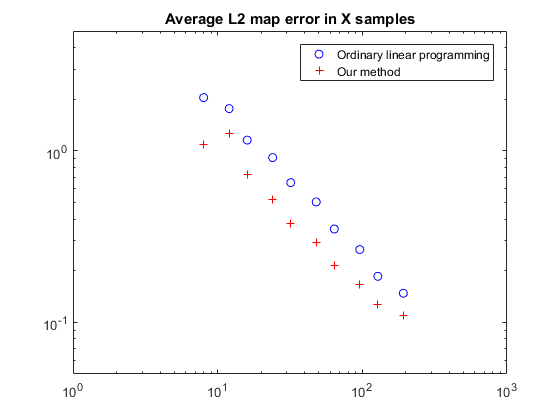}
  \includegraphics[width=6cm,height=5cm]{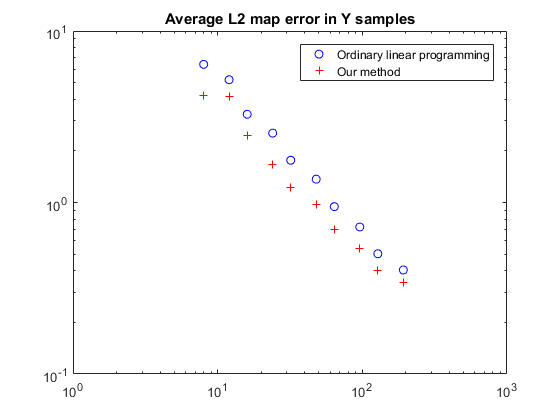}
\caption{Optimal map error in $X$ and $Y$ marginals. The $x$-axis represents the number of cells per dimension.}
\label{figureMixtureMap}
\end{figure}

\subsubsection{A nonlinear map on a Gaussian distribution}

In this example we seek the optimal map between a 2D standard Gaussian distribution $\rho(x)$ and a distribution with density
$$\mu(y) = \frac{9y_1^2y_2^2}{2\pi}e^{-\frac{y_1^6+y_2^6}{2}},$$
for which the optimal map is given by
$$y = (x_1^{1/3},x_2^{1/3})^T.$$
The optimal map error is shown in figure \ref{figureGExample}.

\begin{figure}
\centering
    \includegraphics[width=6cm,height=5cm]{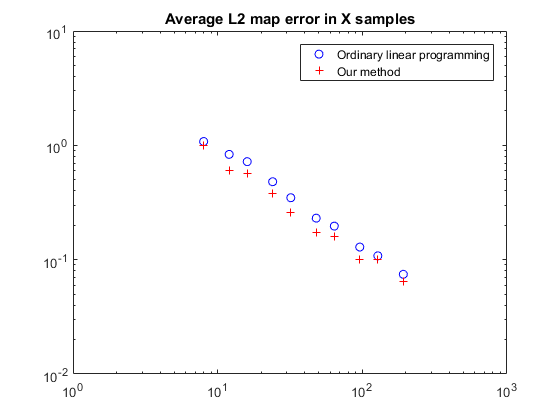}
    \includegraphics[width=6cm,height=5cm]{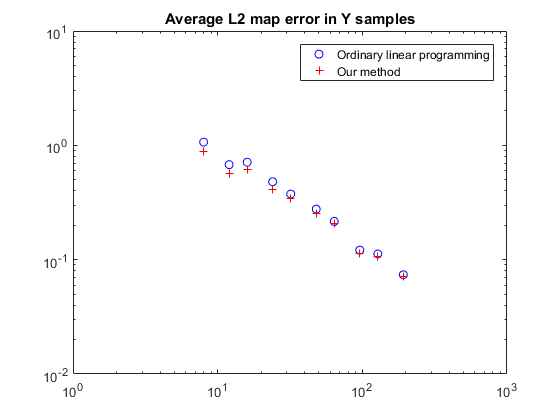}
\caption{Optimal map error for a nonlinear map in the $X$ and $Y$ marginals.}
\label{figureGExample}
\end{figure}

\subsubsection{2D Square to 2D Cross}
The last example that we present finds the optimal plan between a square and a cross, both known through 10000 data points. Both distributions are uniform within their support. The mapped data points are shown in Figure \ref{figureSquareMap}.
\begin{figure}
\centering
  \includegraphics[width=5cm,height=4cm]{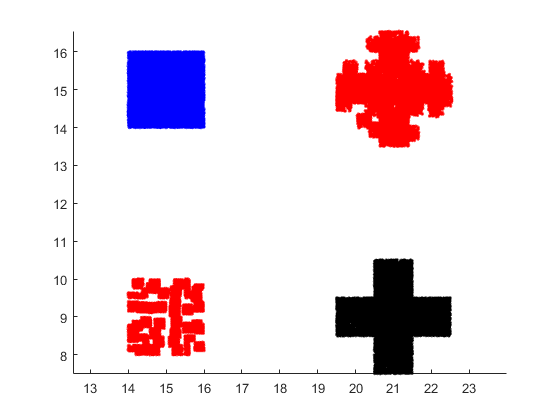}
  \includegraphics[width=5cm,height=4cm]{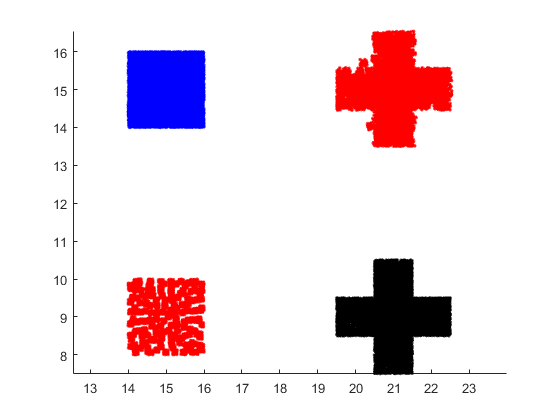}
  \includegraphics[width=5cm,height=4cm]{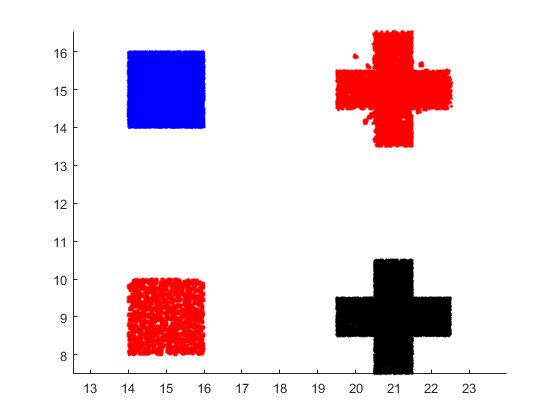}
  \includegraphics[width=5cm,height=4cm]{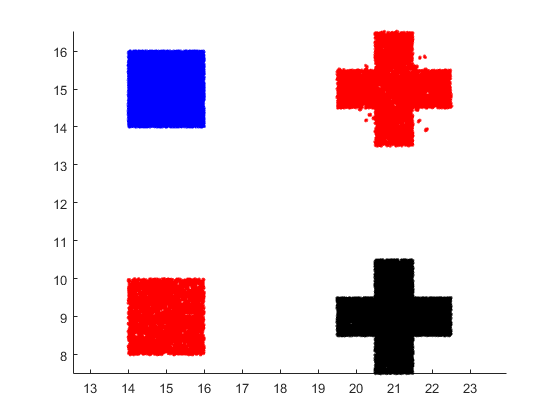}
  \includegraphics[width=5cm,height=4cm]{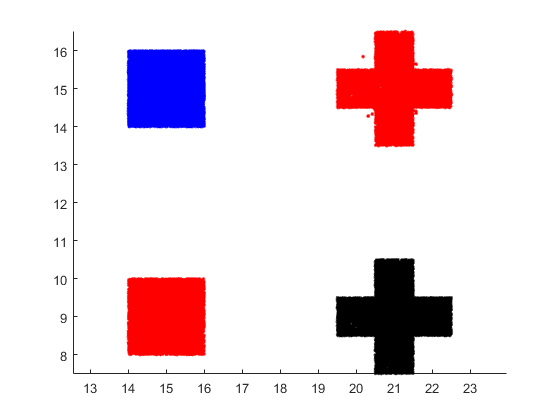}
  \includegraphics[width=5cm,height=4cm]{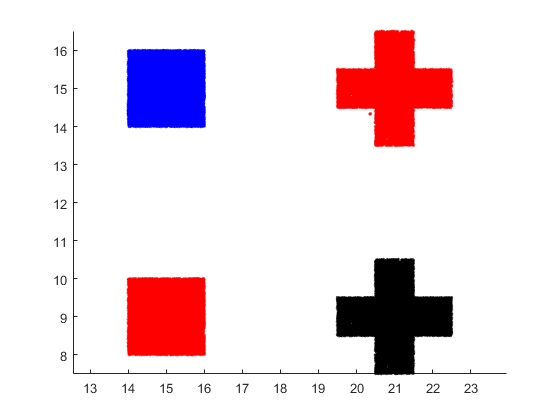}
\caption{Comparison between the target data clouds and recovered data clouds from the numerical solution in each refinement step. Blue data clouds are drawn from the uniform square, black data clouds are drawn from a uniform cross. The red data clouds are recovered from the numerical approximation to the optimal map. The grid size is doubled in each refinement (left to right, top to bottom.)}
\label{figureSquareMap}
\end{figure}

\section{The Wasserstein Barycenter}
\label{ch-2}

The Wasserstein Barycenter problem\cite{Mar10} is a special variant of the optimal transport problem, defined as the minimizer of
\begin{equation}
 (\mathcal{P})\qquad\inf_\nu \sum_{i=1}^p\lambda_i W_2^2(\nu_i,\nu),
\end{equation}
and equivalent to the multi marginal optimal transport problem \cite{Gan98}. The size of a naive discretization of the multi marginal optimal transport problem increases exponentially with the number of marginal distributions, while the size of the Wasserstein barycenter problem increases only linearly. However, solving naively  the resulting large linear programming is costly. Hence we use instead an iterative method that solves the Wasserstein barycenter problem through a set of small pairwise linear programming problems.

\subsection{A numerical scheme for Wasserstein Barycenter}
\label{sec2.2}

Pedro C. Alvariz-Esteban, E.del Barrio, J.A. Cuesta-Albertos and C. Matran (\cite{Ped15}) proposed an effective iterative approach to solve the Wasserstein Barycenter problem.

Following their notation, $\mathcal{P}_{2,ac}(\mathbb{R}^d)$ denotes the set of Borel probabilities on $\mathbb{R}^d$ with finite second moment that are absolutely continuous with respect to the Lebesgue measure, $L(X)$ denotes the probability measure of a random vector $X$, and the optimal transportation maps between probability measure $\nu$ and $\nu_j$ is denoted $T_j$, so
\begin{equation}
L(T_j(X)) = \nu_j, \quad\text{when}\quad L(X) = \nu.
\end{equation}
With this notation, a transformation $G$ is defined through
\begin{equation}
G(\nu) :=L(\sum_{i=1}^p\lambda_iT_i(X)),\quad \text{where}\quad L(X) = \nu.
\end{equation}
Applying $G$ iteratively to an initial distribution converges to the Wasserstein Barycenter \cite{Ped15}.

This iterative procedure can be applied with any numerical method that can solve pairwise optimal transport problems. Especially when combining with our data driven formulation, the iterative procedure becomes very straightforward and easy to implement, since it does not involve any density estimation and Jacobian determinant calculation. The following iterative algorithm describes the data driven formulation of the iterative method.

\begin{algorithm}
\caption{Iterative algorithm to calculate Wasserstein barycenter}
\label{iterativeBarycenter}
Given data samples from $p$ probability measures $\nu_1,\ldots,\nu_p$, denoted $\{x_l^1\}_{l=1}^{n_1},\ldots,\{x_l^p\}_{l=1}^{n_p}$, draw samples $\{y_l^{0}\}_{l=1}^{n_B}$ from any distribution (for instance, $\{y_l^{0}\}_{l=1}^{n_B}=\{x_l^1\}_{l=1}^{n_1}$ ) and repeat the following steps until convergence:
\begin{enumerate}
\item Solve the pairwise data driven optimal transport problem between data samples $\{y_l^{n}\}_{l=1}^{n_B}$ and $\{x_l^i\}_{l=1}^{n_i}$. For each pairwise problem, obtain data points $\{y_l^{i,n+1}\}_{l=1}^{n_B}$ mapped from $\{y_l^{n}\}_{l=1}^{n_B}$ under the corresponding optimal map.
\item Calculate the new barycenter samples $\{y_l^{n+1}\}_{l=1}^{n_B}$ using the fix point update
\begin{equation}
y_l^{n+1} = \sum_{i=1}^p\lambda_iy_l^{i,n+1}
\end{equation}
\end{enumerate}
\end{algorithm}

One practical advantage of this formulation is that step (1) in algorithm \ref{iterativeBarycenter} can be parallellized,  decomposing the barycenter problem into a series of pairwise optimal transport problems. Since the marginal information is provided through samples, it makes sense to return the estimated barycenter also through samples, as in algorithm \ref{iterativeBarycenter}.

\subsection{Numerical Results}
\label{sec2.3}
\subsubsection{Displacement interpolation between two marginal distributions}

We consider again the map between a uniform square and a uniform cross, using algorithm \ref{iterativeBarycenter} to calculate the barycenter under different weights, which is equivalent to the finding the displacement interpolation between the two marginal densities.
The results are shown in Figure \ref{figure_squareCrossBarycenter}.

\begin{figure}
\centering
  \includegraphics[width=4cm,height=3.2cm]{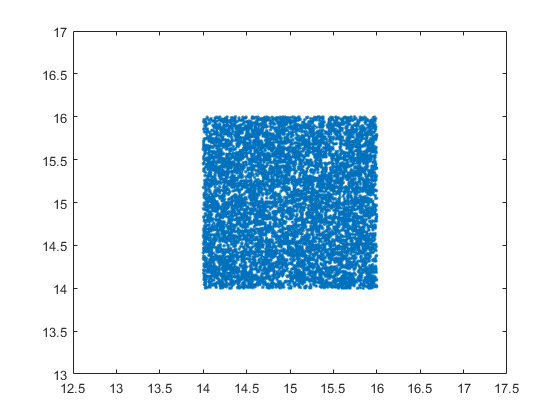}
  \includegraphics[width=4cm,height=3.2cm]{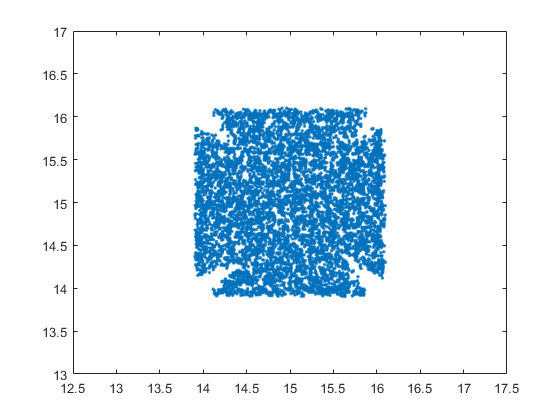}
  \includegraphics[width=4cm,height=3.2cm]{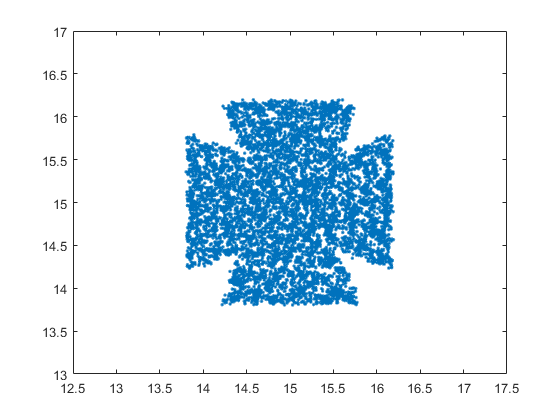}
  \includegraphics[width=4cm,height=3.2cm]{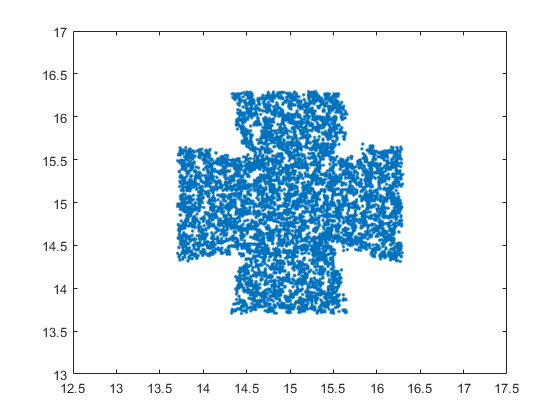}
  \includegraphics[width=4cm,height=3.2cm]{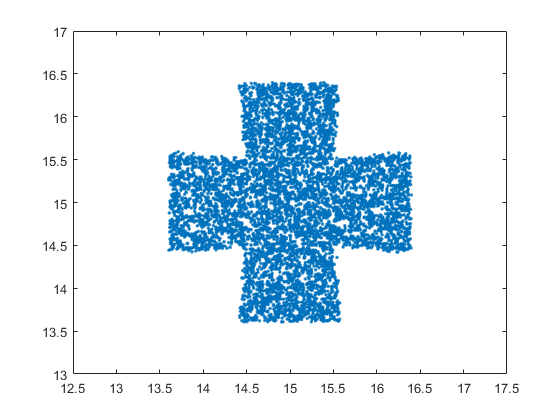}
  \includegraphics[width=4cm,height=3.2cm]{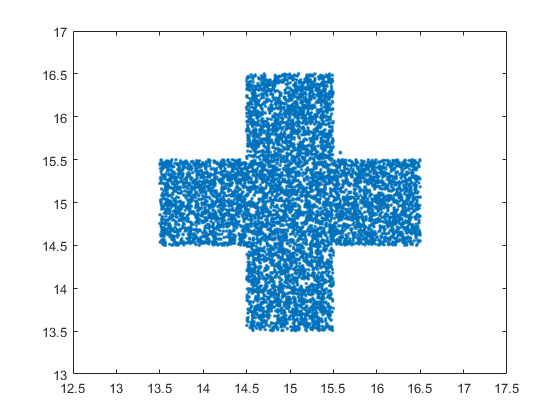}
\caption{Interpolants or weighted barycenters between two distributions, with weights assigned to the first marginal equal to $1,0.8,0.6,0.4,0.2\text{ and }0. $}
\label{figure_squareCrossBarycenter}
\end{figure}

\subsubsection{Wasserstein Barycenter with different weights}

We calculate the Wasserstein barycenters with different weights between three data sets: a circle, a square and a cross. All data samples are drawn from a uniform distribution supported within each shape. The results are shown in Figure \ref{figure_circleSquareCrossBarycenter}.

\begin{figure}
\centering
  \includegraphics[width=13cm,height=13cm]{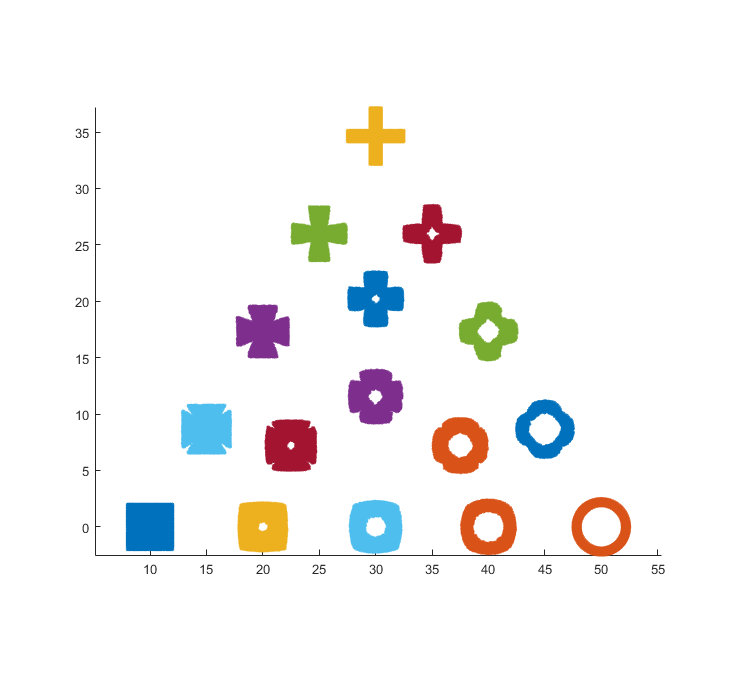}
\caption{Barycenters with different weights between a circle, a square and a cross.}
\label{figure_circleSquareCrossBarycenter}
\end{figure}

\section{Applications}
\label{ch-3}

\subsection{Visualization of perturbed images}
\label{sec3.1}
The Wasserstein Barycenter algorithm is applied in this section to visualize perturbed digit images. We use 550 images from the Chars74K dataset for digits 0-9 (each digit with 55 images). Each image (originally of size $1200\times 900$) is resized to size $512\times512$ and is randomly shifted.

To calculate the Wasserstein barycenter of the 55 images, we use algorithm \ref{iterativeBarycenter} and run the fix point iteration for 3 times (calculating 55 pairwise optimal transportation problems 3 times for each digit). Each pairwise optimal transport problem is run on grids 16, 32, 64,128 and 256, using the grid refinement procedure. All linear programming problems are solved using the commercial software Gurobi. The whole program is run in a parallel computation setting with 12 workers and costs around 23 minutes for each digit. By contrast, it is hard for traditional methods to solve a pairwise optimal transport problem on a $256\times256$ grid alone, not to mention solving a barycenter problem with 55 marginal densities each on a $256\times256$ grid.

The results are shown in figures \ref{figure_applicationDigits1} and \ref{figure_applicationDigits2} (bottom line), compared with the Euclidean barycenter (top line) and the Euclidean barycenter after re-centering images (mid line). The Wasserstein barycenter results are sharper and clearer.

\begin{figure}
\centering
  \includegraphics[width=2.5cm,height=3cm]{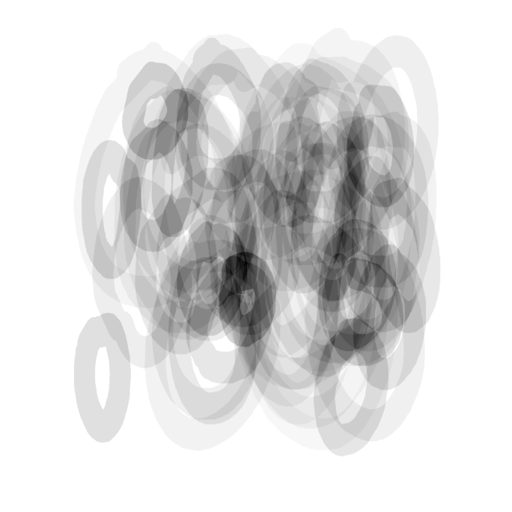}
  \includegraphics[width=2.5cm,height=3cm]{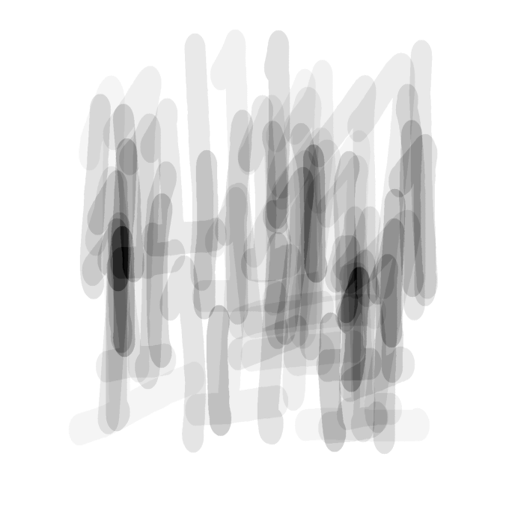}
  \includegraphics[width=2.5cm,height=3cm]{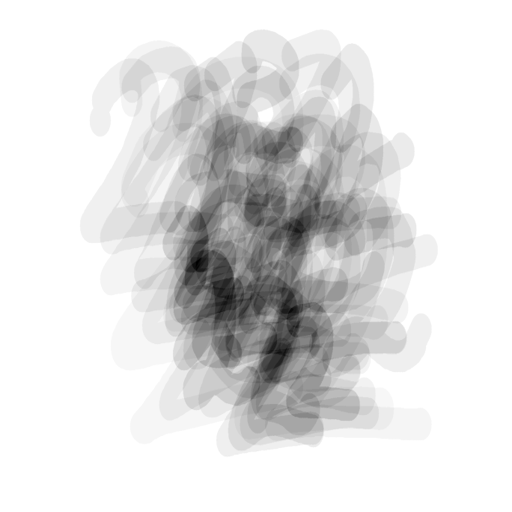}
  \includegraphics[width=2.5cm,height=3cm]{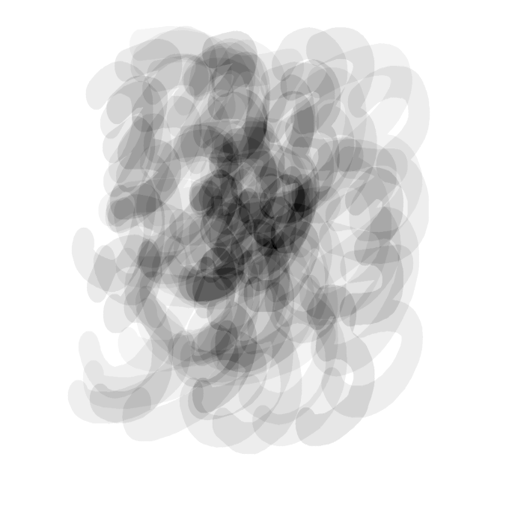}
  \includegraphics[width=2.5cm,height=3cm]{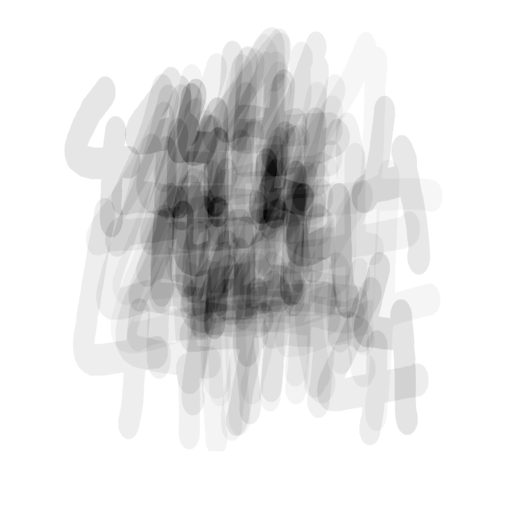}
  \includegraphics[width=2.5cm,height=3cm]{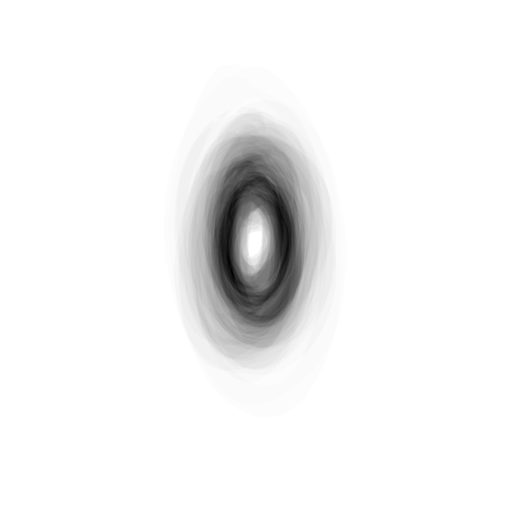}
  \includegraphics[width=2.5cm,height=3cm]{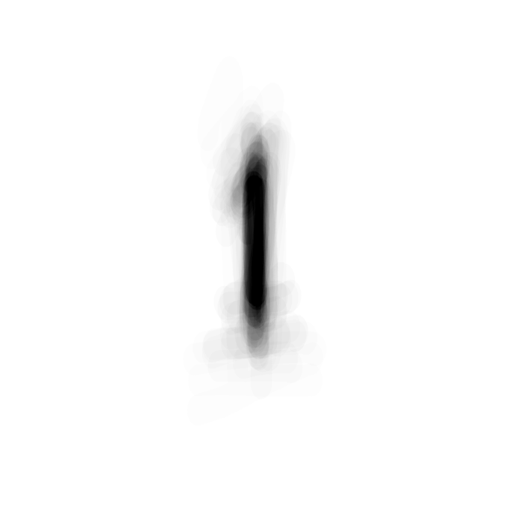}
  \includegraphics[width=2.5cm,height=3cm]{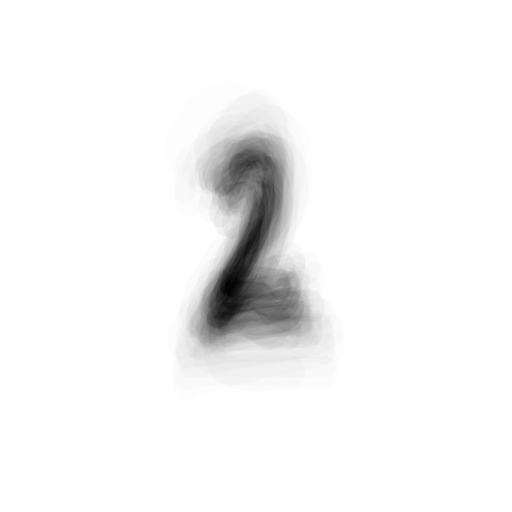}
  \includegraphics[width=2.5cm,height=3cm]{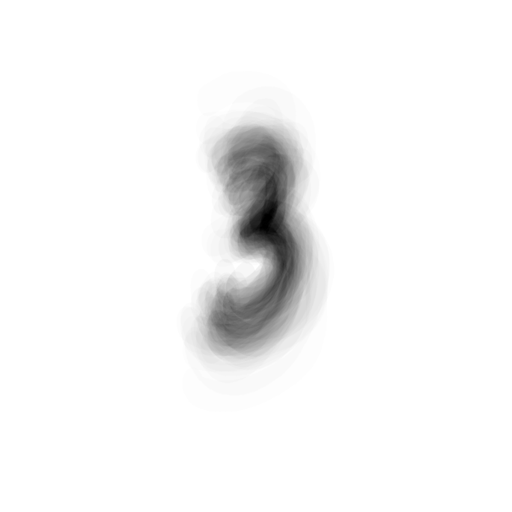}
  \includegraphics[width=2.5cm,height=3cm]{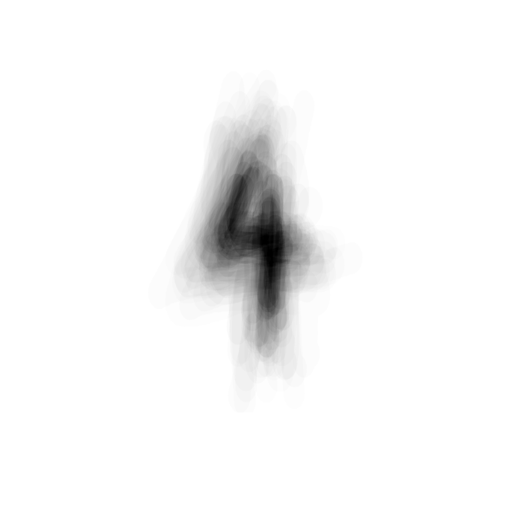}
  \includegraphics[width=2.5cm,height=3cm]{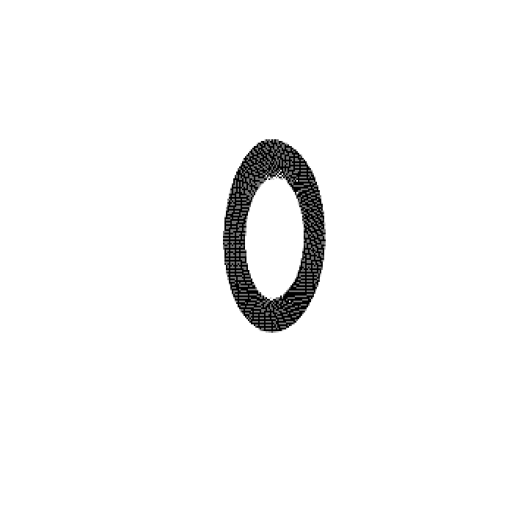}
  \includegraphics[width=2.5cm,height=3cm]{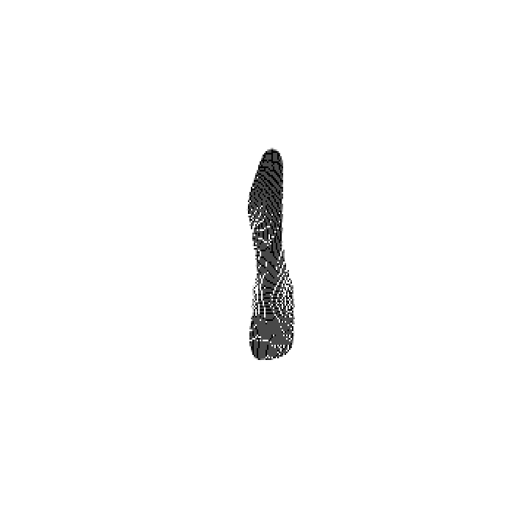}
  \includegraphics[width=2.5cm,height=3cm]{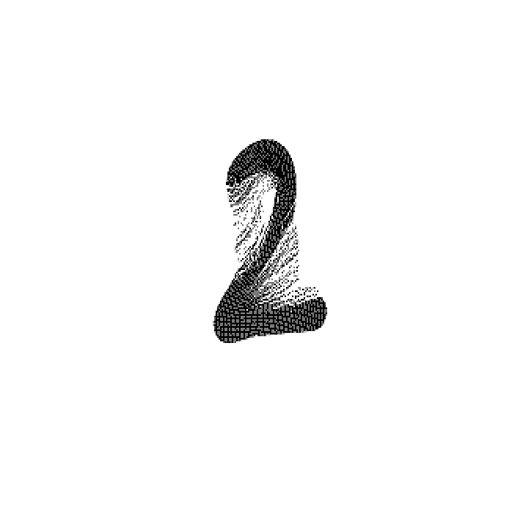}
  \includegraphics[width=2.5cm,height=3cm]{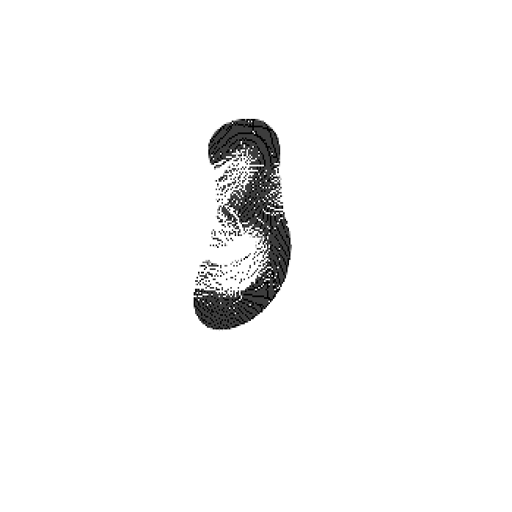}
  \includegraphics[width=2.5cm,height=3cm]{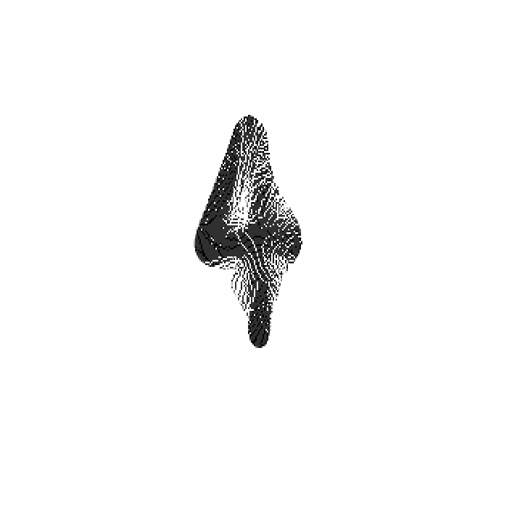}
  \caption{Barycenter of 55 handwritten digits for 0-4. Top row: Euclidean barycenter; Mid row: Euclidean barycenter after re-centering; Bottom row: Wasserstein barycenter.}
\label{figure_applicationDigits1}
\end{figure}
\begin{figure}
\centering
  \includegraphics[width=2.5cm,height=3cm]{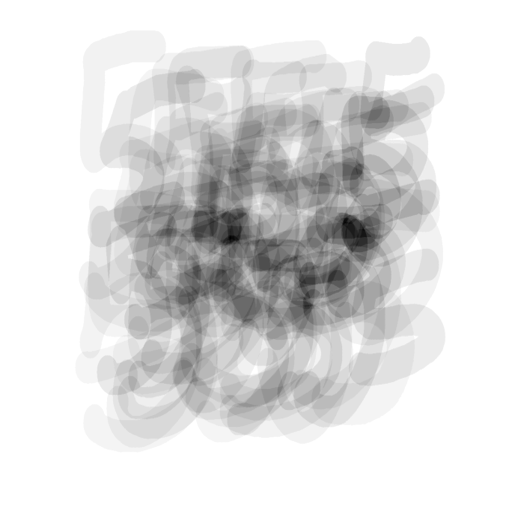}
  \includegraphics[width=2.5cm,height=3cm]{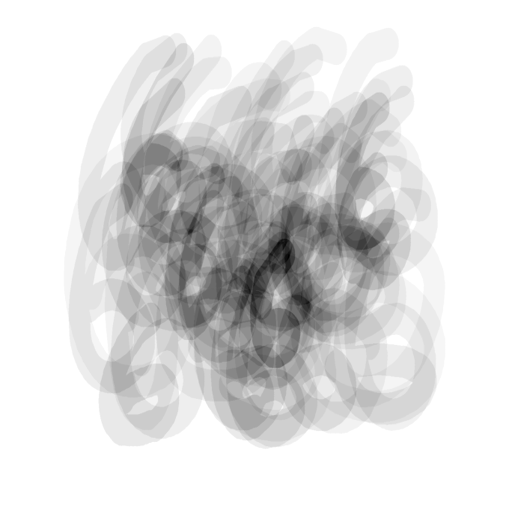}
  \includegraphics[width=2.5cm,height=3cm]{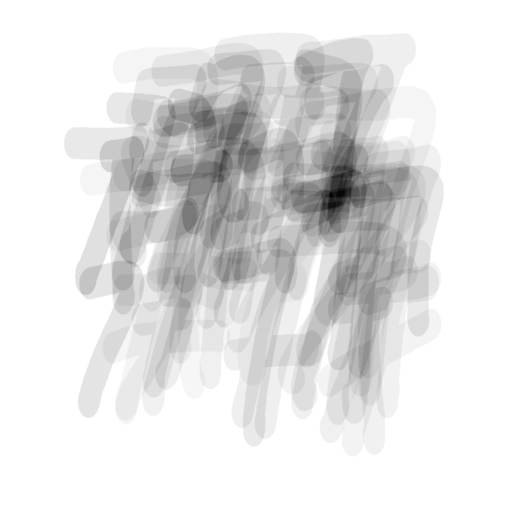}
  \includegraphics[width=2.5cm,height=3cm]{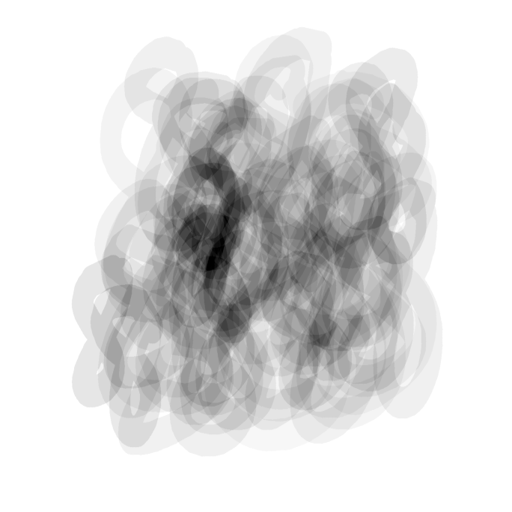}
  \includegraphics[width=2.5cm,height=3cm]{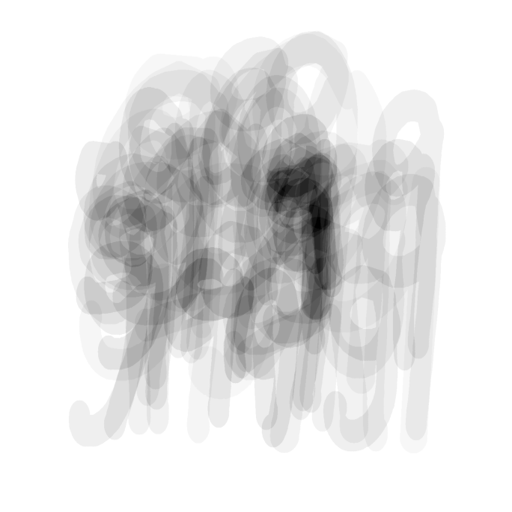}
  \includegraphics[width=2.5cm,height=3cm]{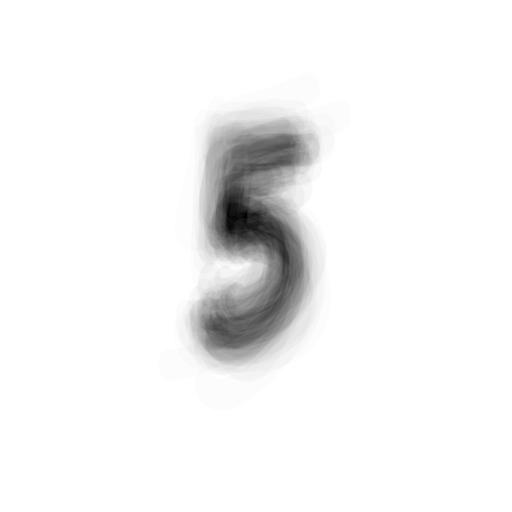}
  \includegraphics[width=2.5cm,height=3cm]{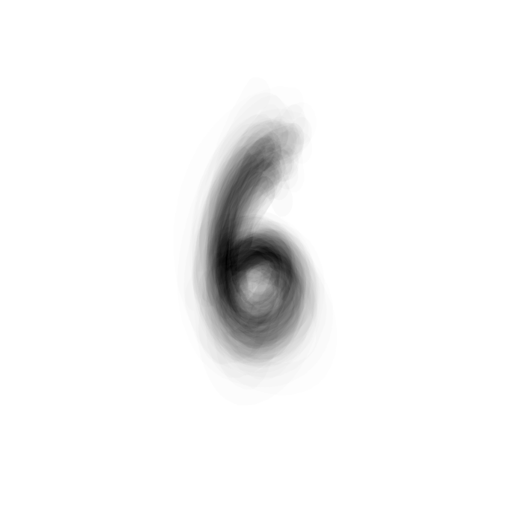}
  \includegraphics[width=2.5cm,height=3cm]{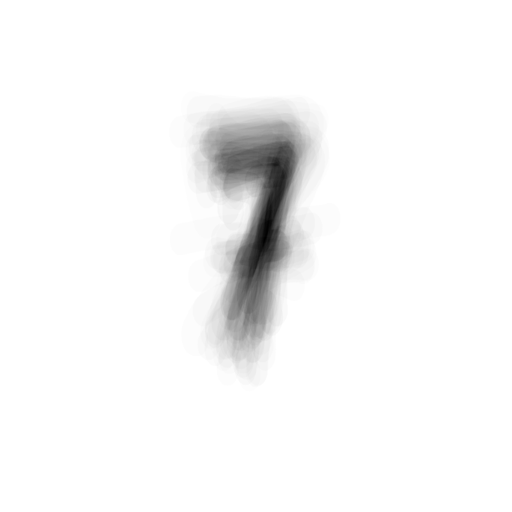}
  \includegraphics[width=2.5cm,height=3cm]{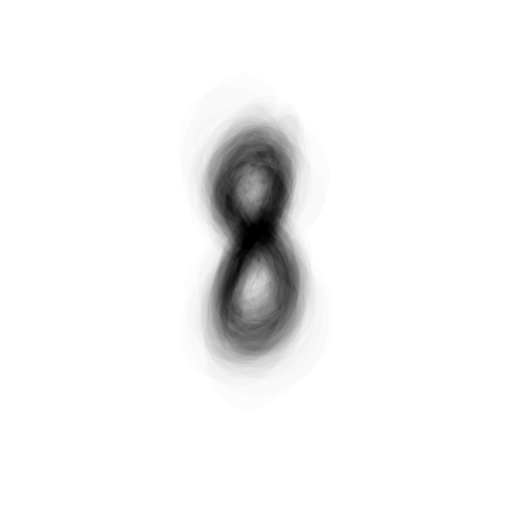}
  \includegraphics[width=2.5cm,height=3cm]{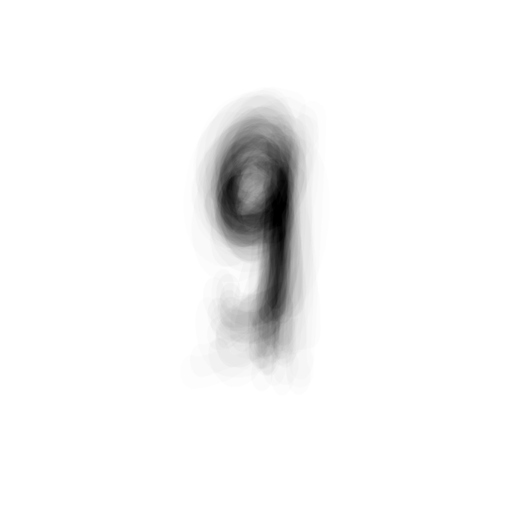}
  \includegraphics[width=2.5cm,height=3cm]{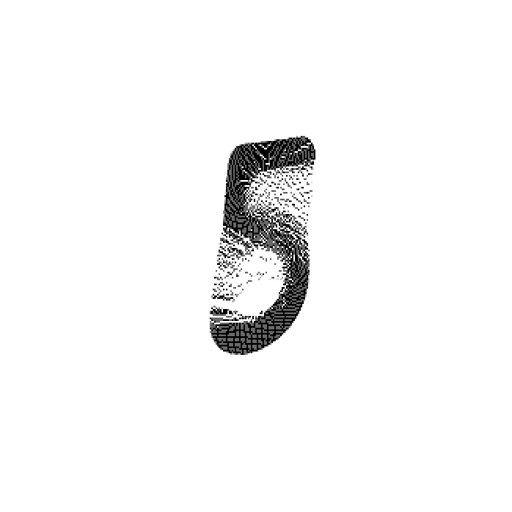}
  \includegraphics[width=2.5cm,height=3cm]{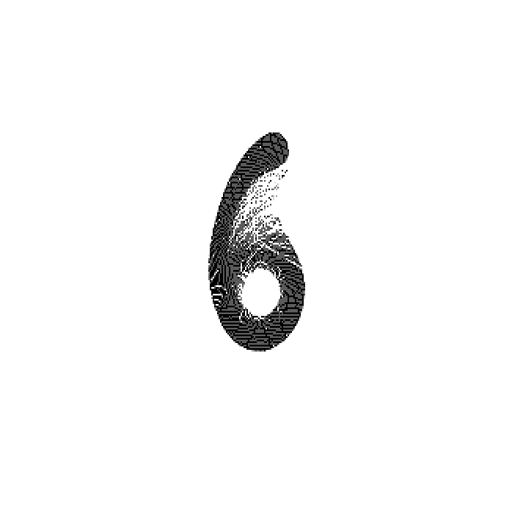}
  \includegraphics[width=2.5cm,height=3cm]{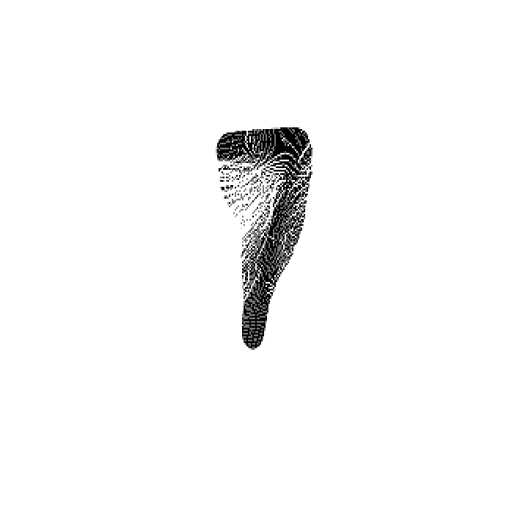}
  \includegraphics[width=2.5cm,height=3cm]{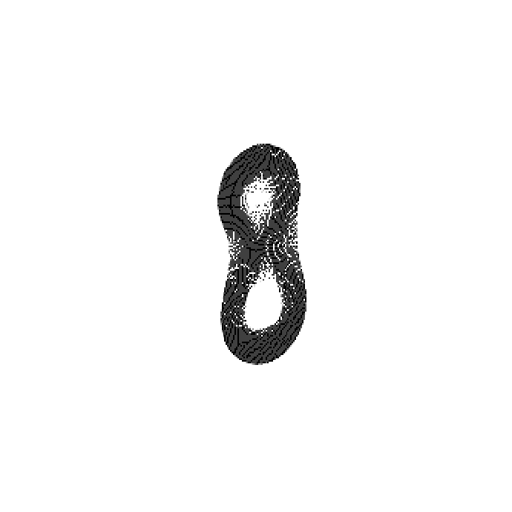}
  \includegraphics[width=2.5cm,height=3cm]{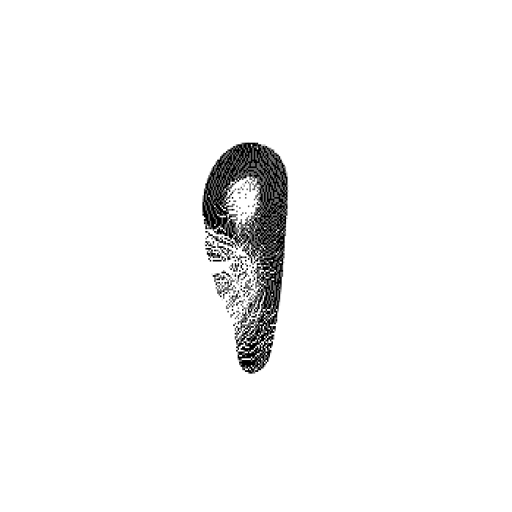}
\caption{Barycenter of 55 handwritten digits for 5-9. Top row: Euclidean barycenter; Mid row: Euclidean barycenter after re-centering; Bottom row: Wasserstein barycenter.}
\label{figure_applicationDigits2}
\end{figure}

\subsection{Texture synthesis}
\label{sec3.2}
Following the framework introduced by Rabin, Peyr\'e, Delon and Bernot (\cite{Rab10}), we apply our optimal transport and barycenter methods to texture synthesis using both first order and high order statistical mixing.

A general texture synthesis procedure usually consists of the following steps:
\begin{enumerate}
\item Use filters to project the image under consideration  onto a feature space.
\item Starting from a random Gaussian noise image, use the same filters to construct the feature space representation for this image. Then map the features of the noise image to the desired image so that the features share the same statistical properties.
\item Recover a new noise image from the mapped features.
\item Enforce the color pixel distribution on the new noise image.
\end{enumerate}
Steps 2,3 and 4 are iterated to enforce the statistical distribution of the feature and color pixel of the desired image.

We adopt a set of steerable pyramid filters with 4 scales and 4 orientations together with a coarse scale frame and a high frequency frame. (That is $4\times4+2=18$ kinds of filters.). Credit is due for the filter implementation to the Center for Neural Science in New York University, as they keep a good repository of Matlab codes for this job.

In this application, our optimal transport method is applied in steps 2 and 4, for constructing a map between two feature distributions/color distributions. The benefits lie in:
\begin{itemize}
\item The feature distribution is known through points, so naturally fit a data driven framework such as ours.
\item The goal of step 2 is to match the statistical distribution of two feature spaces. To put it in optimal transport language, the correctness of the marginal distribution is more important than the optimality of the map. This naturally favors our method as the constraints --with estimated marginals-- are satisfied exactly.
\item The matching distribution can be of very high dimensionality. Implementing a mild adjustment of our refinement algorithm, our methods can be easily applied to a high dimensional setting.
\end{itemize}

Examples of texture synthesis are shown in Figure \ref{figure_firstOrderMixing1}, with the target texture on the left column and the synthesized textures from white noise using first order model and the new pairwise optimal transport solver on the right column.

Results of texture mixing with different weights between two textures are shown in Figure \ref{figure_firstOrderMixing2}. The iterative Wasserstein barycenter algorithm is used to calculate barycenter of each feature. Then the pairwise optimal transport solver is used to map the feature of a random Gaussian texture to the barycenter.

\begin{figure}
\centering
  \includegraphics[width=5cm,height=5cm]{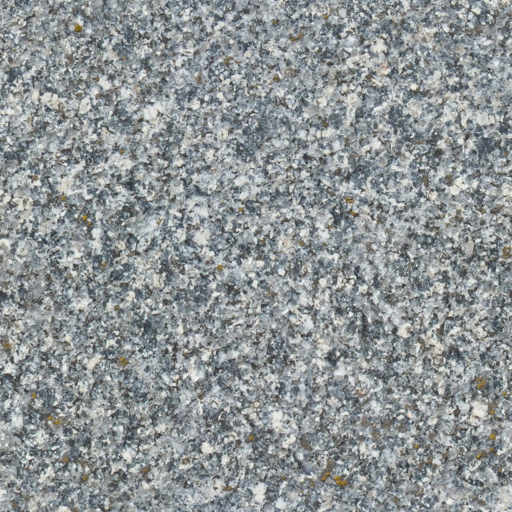}
  \includegraphics[width=5cm,height=5cm]{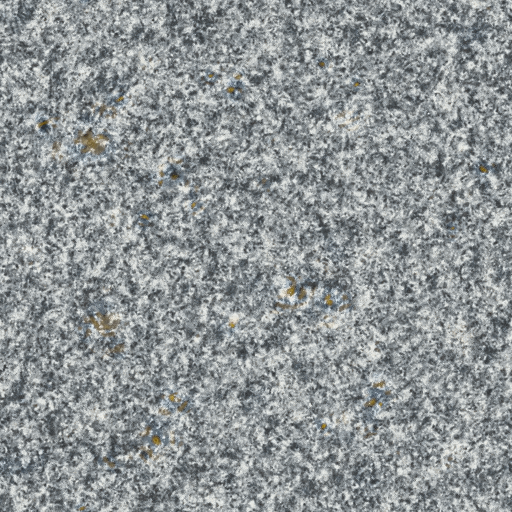}
  \includegraphics[width=5cm,height=5cm]{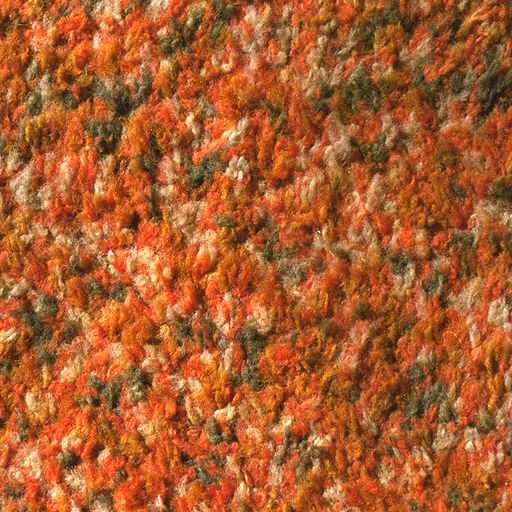}
  \includegraphics[width=5cm,height=5cm]{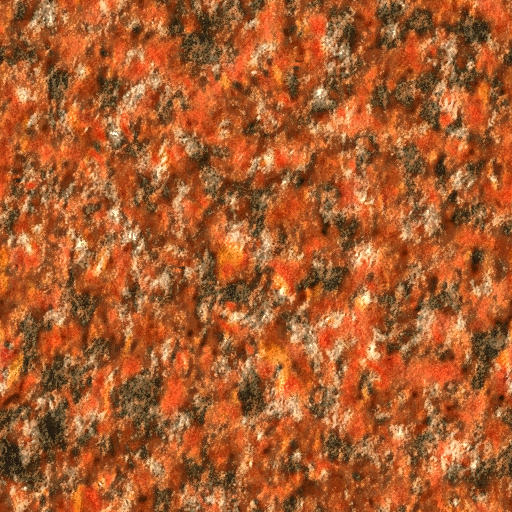}
\caption{Left column: target texture. Right column: synthesized texture from a realization of Gaussian noise using the first order model.}
\label{figure_firstOrderMixing1}
\end{figure}

\begin{figure}
\centering
\begin{tabular}{cccc}
  \subfloat[Original Figure]{\includegraphics[width=2.7cm,height=3cm]{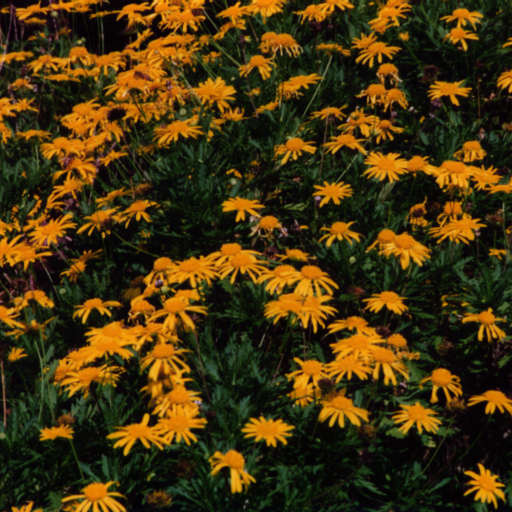}}&
  \subfloat[$\rho=0.1$]{\includegraphics[width=2.7cm,height=3cm]{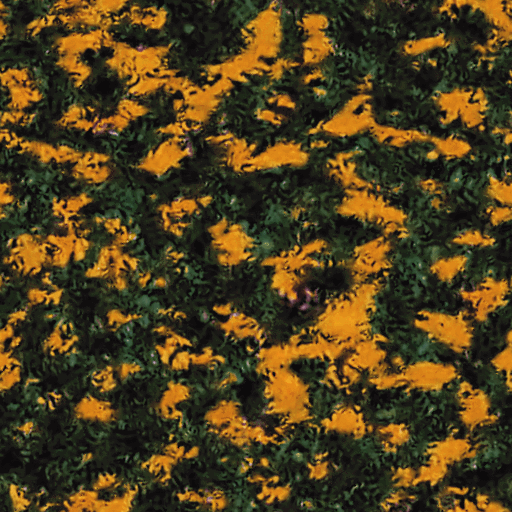}}&
  \subfloat[$\rho=0.2$]{\includegraphics[width=2.7cm,height=3cm]{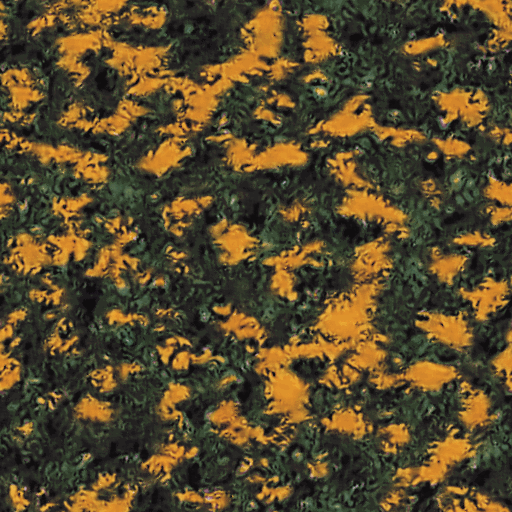}}&
  \subfloat[$\rho=0.3$]{\includegraphics[width=2.7cm,height=3cm]{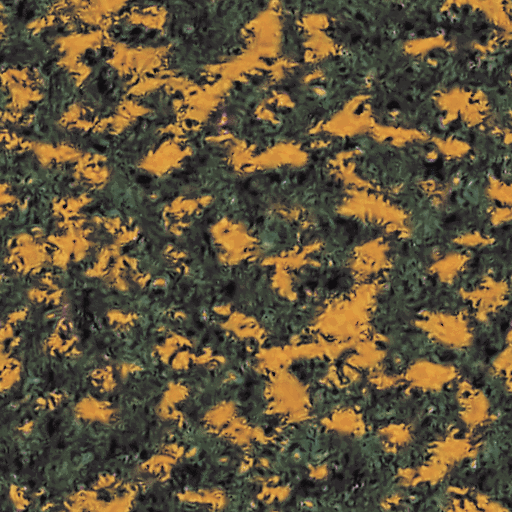}}\\
  \subfloat[$\rho=0.4$]{\includegraphics[width=2.7cm,height=3cm]{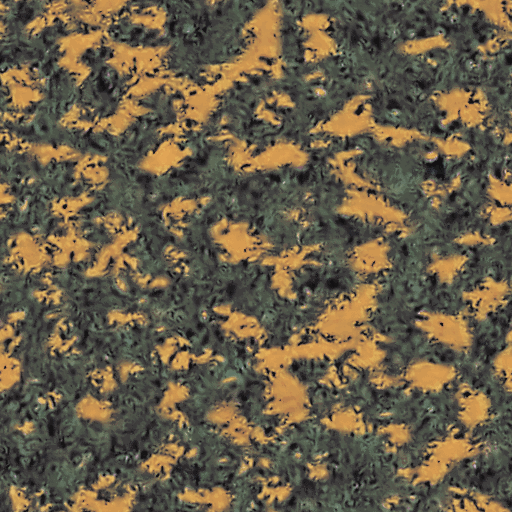}}&
  \subfloat[$\rho=0.5$]{\includegraphics[width=2.7cm,height=3cm]{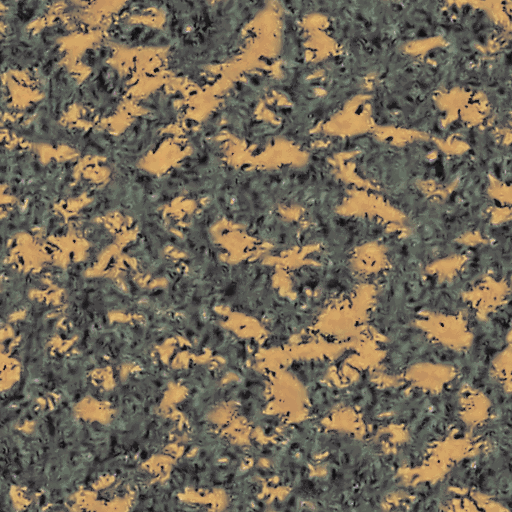}}&
  \subfloat[$\rho=0.6$]{\includegraphics[width=2.7cm,height=3cm]{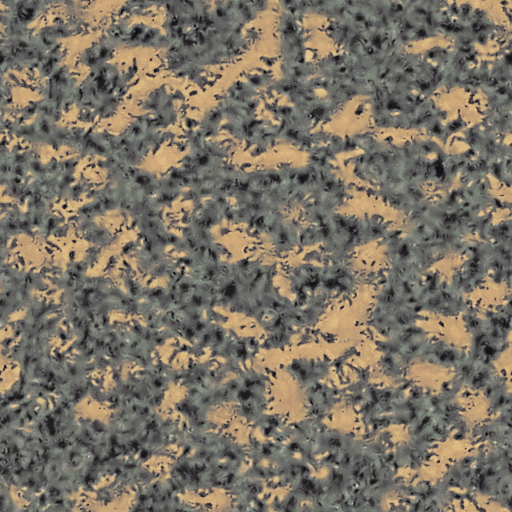}}&
  \subfloat[$\rho=0.7$]{\includegraphics[width=2.7cm,height=3cm]{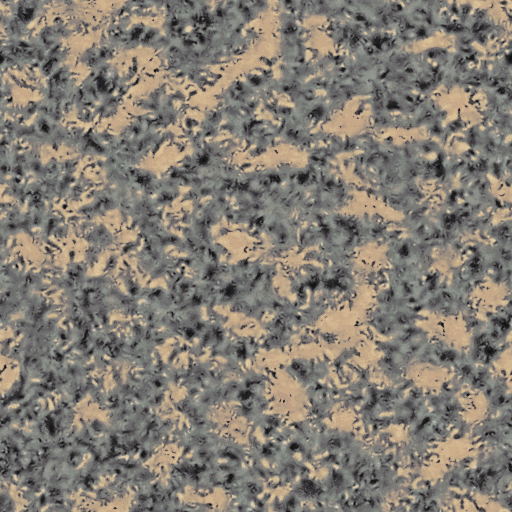}}\\
  \subfloat[$\rho=0.8$]{\includegraphics[width=2.7cm,height=3cm]{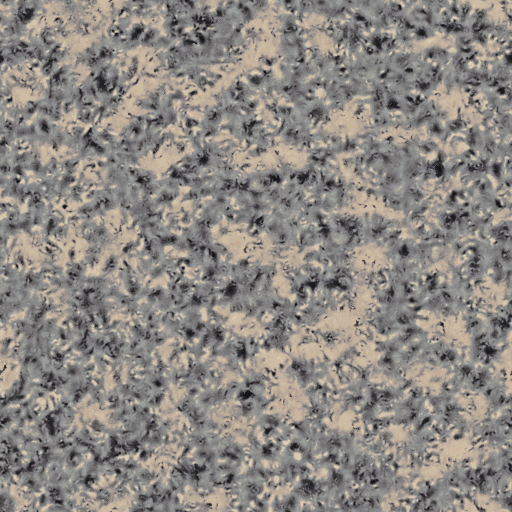}}&
  \subfloat[$\rho=0.9$]{\includegraphics[width=2.7cm,height=3cm]{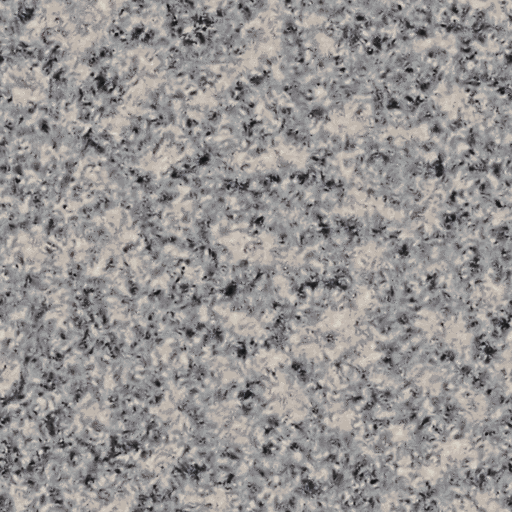}}&
  \subfloat[$\rho=1.0$]{\includegraphics[width=2.7cm,height=3cm]{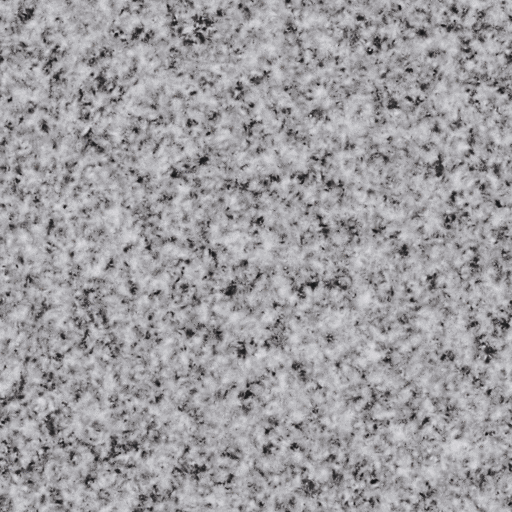}}&
  \subfloat[Original Figure]{\includegraphics[width=2.7cm,height=3cm]{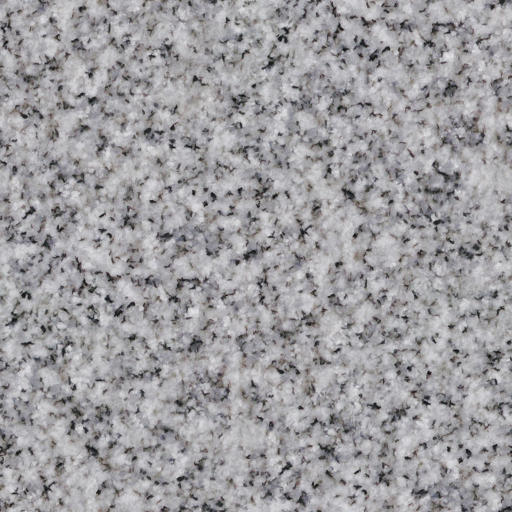}}
\end{tabular}
 \caption{Mixing of two textures using the first order model.}
\label{figure_firstOrderMixing2}
\end{figure}

The high order statistical model introduced by Rabin, Gabriel, Julie and Bernot (\cite{Rab10}) follows the methodology from work of Portilla and Simoncelli (\cite{Por00}) and utilize the joint distribution of local neighbours to explain the spatial correlation. Clustering the feature space into neighbourhoods, the joint distribution has higher dimensionality but fewer data points. For instance, for a $4\times 4$ block neighborhood, the dimensionality becomes $4\times4\times3 = 48$, while the number of data points decreases by a factor of $1/16$.

To deal with high dimensional optimal transport, we make some adjustments to our refinement algorithm:
\begin{enumerate}
\item To initialize the algorithm, we do not discretize every dimension but only the one with longest support (For instance, divide the data cloud into two parts.)
\item In each iteration, we no longer cut each dimension into a half but only refine the dimension with longest support, thus in each iteration the number of marginal variables is doubled.
\end{enumerate}
Since the function space in each iteration is constrained, the size of the linear programming problem still grows linearly with the number of marginal variables. Though the number of marginal variables grows exponentially, it does so with a constant factor 2 regardless of the dimensionality of the data.

Figure \ref{figure_highOrder} shows some examples of using the higher order model with different neighborhoods. Compared with the first order model, the higher order model requires computation of high dimensionality but also captures more detailed structures in the texture.

\begin{figure}
\centering
  \includegraphics[width=3cm,height=3cm]{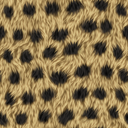}
  \includegraphics[width=3cm,height=3cm]{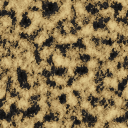}
  \includegraphics[width=3cm,height=3cm]{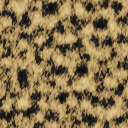}
  \includegraphics[width=3cm,height=3cm]{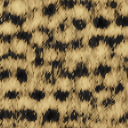}\\
  \vspace{0.1cm}
  \includegraphics[width=3cm,height=3cm]{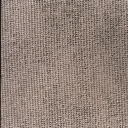}
  \includegraphics[width=3cm,height=3cm]{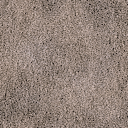}
  \includegraphics[width=3cm,height=3cm]{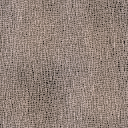}
  \includegraphics[width=3cm,height=3cm]{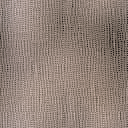}\\
  \vspace{0.1cm}
  \includegraphics[width=3cm,height=3cm]{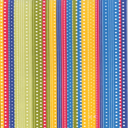}
  \includegraphics[width=3cm,height=3cm]{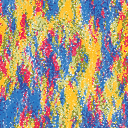}
  \includegraphics[width=3cm,height=3cm]{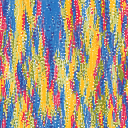}
  \includegraphics[width=3cm,height=3cm]{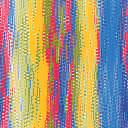}\\
  \vspace{0.1cm}
  \includegraphics[width=3cm,height=3cm]{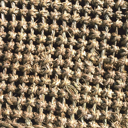}
  \includegraphics[width=3cm,height=3cm]{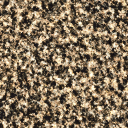}
  \includegraphics[width=3cm,height=3cm]{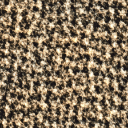}
  \includegraphics[width=3cm,height=3cm]{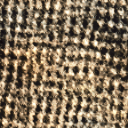}\\
  \vspace{0.1cm}
  \includegraphics[width=3cm,height=3cm]{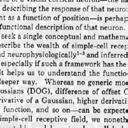}
  \includegraphics[width=3cm,height=3cm]{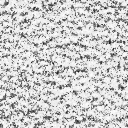}
  \includegraphics[width=3cm,height=3cm]{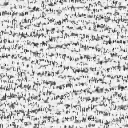}
  \includegraphics[width=3cm,height=3cm]{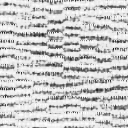}
  \caption{High order synthesis. From left to right: target texture, first order model, $4\times4$ and $8\times8$ neighbours using high order model}
\label{figure_highOrder}
\end{figure}

\section{Summary}
This article develops a linear-programming-based data driven methodology to solve the optimal transportation problem though a sequence of refined meshes. It extends the work in \cite{Adam15}, achieving a better approximation to the solution with almost no additional effort, by approximating the marginal distributions through mixtures, where each component has support in one rectangular cell. In addition, these components are factorized into products of one-dimensional distributions, for which a) the density estimation is straightforward, and b) the pairwise optimal transport between individual $x$ and $y$ components can be solved in closed form.
The new method requires only very elementary, one-dimensional and local density estimations. It involves no additional time dimension or partial differential equation to solve, and it can be calculated fast by constraining the function space in each step.
An adaptively refined mesh solves the problem in multi grids with a number of unknwon that grows only linearly with the size of the discretization of the marginals.


For the Wasserstein barycenter problem, we apply the iterative approach of \cite{Ped15} and combine it with the new pairwise optimal transport problem solver. The resulting data-driven algorithm is naturally parallelizable, making possible the calculation of the barycenter for large number of marginal densities.

The metholology is illustrated through two applications: the visualization of blurred images, and texture mixing using both first order and higher order statistical models.

\clearpage
\bibliographystyle{spmpsci}      
\bibliography{Chen-Tabak}   

\end{document}